\documentclass[12pt]{amsart}
\usepackage{amsmath, amssymb, latexsym}
\usepackage{mathrsfs}
\usepackage{enumerate}
\usepackage{color, url}
\usepackage{bbm}
\usepackage[colorlinks=true,linkcolor=blue,citecolor=purple]{hyperref}
\usepackage{tikz}
\usetikzlibrary{positioning}
\numberwithin{equation}{section}
\newtheorem{theorem}{Theorem}[section]
\newtheorem{corollary}[theorem]{Corollary}
\newtheorem{lemma}[theorem]{Lemma}
\newtheorem{proposition}[theorem]{Proposition}

\newtheorem{example}[theorem]{Example}
\theoremstyle{definition}
\newtheorem{definition}[theorem]{Definition}

\newcommand{\Hom}{\operatorname{Hom}}

\newcommand{\A}{\mathbf{A}}
\newcommand{\B}{\mathbf{B}}

\newcommand{\Q}{\mathbf{Q}}

\newcommand{\Z}{\mathbf{Z}}
\newcommand{\h}{\mathfrak{h}}
\newcommand{\g}{\mathfrak{g}}
\allowdisplaybreaks

\title[Perfect basis theory for quantum Borcherds-Bozec algebras]
{Perfect basis theory for quantum \\ Borcherds-Bozec algebras}

\author[Zhaobing Fan]{Zhaobing Fan}
\address{Harbin Engineering University,
Harbin, China}
\email{fanzhaobing@hrbeu.edu.cn}
\thanks{}

\author[Shaolong Han]{Shaolong Han${}^*$}
\address{Beijing International Center for Mathematical Research, Peking University, Beijing 100871, China}
\email{algebra@hrbeu.edu.cn}

\author[Seok-Jin Kang]{Seok-Jin Kang}
\address{Korea Research Institute of Arts and Mathematics,
Asan-si, Chungcheongnam-do, 31551, Korea}
\email{soccerkang@hotmail.com}
\thanks{}

\author[Young Rock Kim]{Young Rock Kim}%${}^{*}$}
\address{Graduate School of Education, Hankuk University of Foreign Studies, Seoul, 02450,  Korea}
\email{rocky777@hufs.ac.kr} %
\thanks{${}^{*}$ Corresponding author. All authors contribute equally.}

\begin{document}



\begin{abstract}
In this paper, we develop the perfect basis theory for quantum Borcherds-Bozec algebras $U_{q}(\g)$ and their irreducible highest weight modules $V(\lambda)$. We show that the lower perfect graph (resp. upper perfect graph) of every lower perfect basis (resp. upper perfect basis) of $U_{q}^{-}(\g)$ (resp. $V(\lambda)$) is isomorphic to the crystal $B(\infty)$ (resp. $B(\lambda)$).

\end{abstract}

\maketitle

\tableofcontents
\section{Introduction}
Let $G$ be a simple simply-connected
complex algebraic group and $L(\lambda)$ be the irreducible representation of $G$ with 
dominant integral weight $\lambda$.
The research of perfect bases can be traced back to the work of Gelfand and Zelevinsky in \cite{GZ85,GZ86,GZ96}, where they proposed good bases for $L(\lambda)$ which are useful for describing the tensor product multiplicities combinatorially.
To study the geometric  crystals of the Borel subgroup $B$ of $G$, Berenstein and Kazhdan introduced the notion of perfect bases which can be considered as the enhanced version of the good bases {\color{red} \cite{BeK}}.
They took a geometric approach to the combinatorial objects and constructed the perfect bases for integrable highest weight modules over Kac-Moody algebras.
Thus it gives rise to a natural crystal structure without taking the quantum deformation and
crystal limit.

\vskip 2mm

In \cite{KOP11} and \cite{KKKS}, Berenstein-Kazhdan theory was extended to the
perfect basis and dual perfect basis theory for quantum Borcherds algebras and
their irreducible highest weight modules.
As is the case with Berenstein-Kazhdan theory, the crystals arising from perfect
bases and dual perfect bases are isomorphic to either $B(\infty)$,
the crystal of  the negative half  of the quantum Borcherds algebra,
or to $B(\lambda)$, the crystal of the irreducible highest weight module, respectively.
Moreover, they showed that the perfect basis theory plays an important role in the
categorification of quantum Borcherds algebras and their highest weight modules {\color{red} \cite{KOP12}}.

\vskip 2mm

As we have seen in the past 25 years, one of the most significant achievements in representation theory of quantum groups have appeared: Lusztig's canonical basis theory \cite{Lus90, Lus91} and Kashiwara's crystal
basis theory  \cite{Kashi90, Kashi91}.

\vskip 2mm

The canonical basis theory is closely related to the theory of perverse sheaves on
the representation varieties of quivers without loops.
In  \cite{Bozec2014b, Bozec2014c}, Bozec extended Lusztig's theory to the
study of perverse sheaves and quivers with loops, thereby introduced the notion
of quantum Borcherds-Bozec algebras.
The quantum Borcherds-Bozec algebras can be viewed as a huge generalization of
quantum groups and quantum Borcherds algebras \cite{Dr85, Jimbo85, Kang95}.
The theory of highest weight module, Ringel-Hall algebra construction, canonical bases, crystal bases for
quantum Borcherds-Bozec algebras have been
developed and investigated in  \cite{KK20, Kang18, Lu21, FHKK}.
We expect there are much more to be explored in the theory of
quantum Borcheds-Bozec algebras from various points of view.

\vskip 2mm

In this paper, we develop the theory of perfect bases for quantum Borcherds-Bozec algebras
and their integrable highest weight modules.
Let $U_{q}^{-}(\g)$ be the negative part of the quantum Borcherds-Bozec algebra $U_{q}(\g)$
and let $V(\lambda)$ be the irreducible highest weight module with $\lambda \in P^{+}$.
We first define the notion of {\it lower perfect bases}  and  {\it lower perfect graphs}
for weighted  vector spaces associated to a
Borcherds-Cartan datum.
%\vskip 2mm
Then we show that the lower global basis $\mathbf{B}(\infty)$ (resp. $\mathbf{B}(\lambda)$)
of $U_{q}^{-}(\g)$ (resp. $V(\lambda)$)
is a lower perfect basis  of $U_{q}^{-}(\g)$ (resp. $V(\lambda)$), which gives a lower perfect graph structure on $U_{q}^{-}(\g)$ (resp. $V(\lambda)$). 

\vskip 2mm

One of the heart of perfect basis theory is that the perfect graphs have natural crystal structure. For this purpose, we prove the uniqueness of lower perfect graphs in the lower perfect bases. This fact implies that the lower perfect graphs of $U_{q}^{-}(\g)$ (resp. $V(\lambda)$) are all isomorphic to the crystal $B(\infty)$ (resp. $B(\lambda)$).

\vskip 2mm

As the dual version of lower perfect basis, we introduce the notion of  {\it upper perfect basis} and  {\it upper perfect graph} for a weighted space associated with Borcherds-Cartan datum. We show that the lower perfect basis and upper perfect basis can be determined by each other. Since we consider the restricted dual space of $U^-_q(\g)$ and $V(\lambda)$, by the Kashiwara's bilinear form, the restricted dual spaces of $U^-_q(\g)$ and $V(\lambda)$ can be identified with  the original spaces. Thus the upper perfect bases of $U^-_q(\g)$ and $V(\lambda)$ share some of similar
properties with  lower perfect bases.

\vskip 2mm

The theory of perfect bases has many applications to representation theory.
In \cite{KKO13, KKO14}, Kang, Kashiwara and Oh used the strong perfect bases theory to characterize the $\mathbb A$-forms of the highest weight module $V(\Lambda)$ of quantum Kac-Moody algebra and obtained a supercategorification of $V(\Lambda)$.
The perfect basis theory developed here will take a significant part in our on-going work on
categorification of quantum Borcherds-Bozec algebras and their highest weight modules.

%Let $G$ be a simple simply-connected
%complex algebraic group and let $N$ be its unipotent subgroup. To study the bases of the coordinate ring $\mathbb C[N]$, Baumann, Kamnitzer and Knutson introduced the notion of biperfect bases and they proved that the Mirkovi\' c--Vilonen basis of $\mathbb C[N]$ is biperfect (cf. \cite{BKK}).  

\vskip 2mm

This paper is organized as follows. 
In Section \ref{sec:qBB}, we recall the definition of quantum Borcherds-Bozec algebra $U_q(\g)$ and its representation theory.
In Sections \ref{sec:lcb}, we review the crystal basis theory for quantum Borcherds-Bozec algebra and its highest weight modules.
In Section \ref{sec:lgb}, we give the construction of lower global bases $\mathbf B(\infty)$ and $\mathbf B(\lambda)$  for $U^-_q(\g)$ and $V(\lambda)$, respectively.
In Section \ref{sec:lpb}, we introduce the lower perfect basis for a weighted space and show that the lower global bases $\mathbf B(\infty)$ and $\mathbf B(\lambda)$ are lower perfect bases and their lower perfect graphs are isomorphic to the crystals $B(\infty)$ and $B(\lambda)$, respectively. In Section \ref{sec:unique}, we prove the uniqueness of lower perfect graphs of lower perfect bases.  In Section \ref{sec:upper}, we introduce the upper perfect basis and show that the  dual basis of an upper (resp. lower) perfect basis is a lower (resp. upper) perfect basis. Therefore, the upper global basis $\mathbf B(\infty)^\vee$ (resp. $\mathbf B(\lambda)^\vee$)  of $U^-_q(\g)$ (resp. $V(\lambda)$) is an upper perfect basis and its upper perfect  graph is isomorphic to $B(\infty)$ (resp. $B(\lambda)$).

\vskip 3mm 

%\noindent
%{\bfseries  Acknowledgments.}

\noindent
{\it Acknowledgements}.

\vskip 2mm

Z. Fan was partially supported by the NSF of China grant 12271120 and the Fundamental Research Funds for the central universities. S.-J. Kang was supported by China grant YZ2260010601. Y. R. Kim was supported by the National Research Foundation of Korea (NRF) grant funded by the Korea government (MSIT) (No. 2021R1A2C1011467) and was supported by Hankuk University of Foreign Studies Research Fund.

\section{Quantum Borcherds-Bozec algebras}
\label{sec:qBB}

Let $I$ be an index set possibly countably infinite. An
integer-valued matrix $A=(a_{ij})_{i,j \in I}$ is called an {\it
	even symmetrizable Borcherds-Cartan matrix} if it satisfies the
following conditions:

\begin{itemize}
	\vskip 2mm
	\item[(i)] $a_{ii}=2, 0, -2, -4, ...$,
	\vskip 2mm	
	\item[(ii)] $a_{ij}\le 0$ for $i \neq j$,
	\vskip 2mm	
	\item[(iii)] $a_{ij}=0$ if and only if $a_{ji}=0$,
	\vskip 2mm	
	\item[(iv)] there exists a diagonal matrix $D=\text{diag} (s_{i} \in
	\Z_{>0} \mid i \in I)$ such that $DA$ is symmetric.
\end{itemize}

\vskip 2mm

Set $I^{\text{re}}=\{i \in I \mid a_{ii}=2 \}$,
$I^{\text{im}}=\{i \in I \mid a_{ii} \le 0\}$ and
$I^{\text{iso}}=\{i \in I \mid a_{ii}=0 \}$.

\vskip 3mm

A {\it Borcherds-Cartan datum} consists of :
	\vskip 2mm
\ \ (a) an even symmetrizable Borcherds-Cartan matrix
$A=(a_{ij})_{i,j \in I}$,
	\vskip 2mm
\ \ (b) a free abelian group $P$, the {\it weight lattice},
	\vskip 2mm
\ \ (c) $\Pi=\{\alpha_{i} \in P  \mid i \in I \}$, the set of {\it
	simple roots},
	\vskip 2mm

\ \ (d) $P^{\vee} := \Hom(P, \Z)$, the {\it dual weight lattice},
	\vskip 2mm
\ \ (e) $\Pi^{\vee}=\{h_i \in P^{\vee} \mid i \in I \}$, the set of
{\it simple coroots}

\vskip 2mm

\noindent satisfying the following conditions

\vskip 2mm

\begin{itemize}
	\vskip 2mm	
	\item[(i)] $\langle h_i, \alpha_j \rangle = a_{ij}$ for all $i,
	j \in I$,
	\vskip 2mm	
	\item[(ii)] $\Pi$ is linearly independent,
	\vskip 2mm	
	\item[(iii)] for each $i \in I$, there exists an
	element $\Lambda_{i} \in P$ such that $$\langle h_i, \Lambda_j
	\rangle = \delta_{ij} \ \ \text{for all} \ i, j \in I.$$
\end{itemize}

\vskip 1mm

For a given Borcherds-Cartan matrix, there always exists such
a Borcherds-Cartan datum. The elements $\Lambda_i$'s $(i \in
I)$ are called the {\it fundamental weights}.

\vskip 3mm

We denote by
$$P^{+}:=\{\lambda \in P \mid \langle h_i, \lambda \rangle \ge 0 \
\text{for all} \ i \in I \}$$ the set of {\it dominant integral
	weights}. The free abelian group $R:= \bigoplus_{i \in I}
\Z \alpha_i$ is called the {\it root lattice}. Let $R_{+}
= \sum_{i\in I} \Z_{\ge 0} \alpha_i$ be the positive cone of the root lattice. For $\beta = \sum k_i \alpha_i
\in R_{+}$, we define its {\it height} to be
$|\beta|:=\sum k_i$.

\vskip 3mm

Set $\mathfrak{h} = \mathbf C \otimes P^{\vee}$. For $\lambda, \mu \in
{\mathfrak h}^{*}$, we define a partial ordering by $\mu \leq
\lambda$ if and only if $\lambda - \mu \in R_{+}$. Since $A$
is symmetrizable, there exists a non-degenerate symmetric bilinear
form $( \ , \ )$ on ${\mathfrak h}^{*}$ satisfying
$$(\alpha_{i}, \lambda) = s_{i} \langle h_{i}, \lambda \rangle \quad
\text{for all} \ \ \lambda \in {\mathfrak h}^{*}.$$

Let $q$ be an indeterminate.
For $i\in I$ and $n \in \mathbf Z_{> 0}$, let $x_i$ be a symbol, we define
$$q_{i}  = q^{s_i}, 
\quad [n]_{i} = \dfrac{q_{i}^{n} - q_{i}^{-n}} {q_{i} - q_{i}^{-1}}, \quad [n]_{i} ! = [n]_{i} [n-1]_{i} \cdots [1]_{i},\quad x^{(n)}_i=\frac{x^n_i}{[n]_{i} !}.$$

\vskip 3mm

\begin{definition}
	The {\it quantum Borcherds-Bozec algebra} $U_{q}(\g)$ is the
	associative algebra over $\Q(q)$ with $\mathbf{1}$ generated by
	$q^{h}$ $(h \in P^{\vee})$, ${\mathtt a}_{il}$, ${\mathtt b}_{il}$ $((i,l) \in I^{\infty})$
	with the defining relations
	\begin{equation} \label{eq:primitive}
		\begin{aligned}
			& q^0=\mathbf 1,\quad q^hq^{h'}=q^{h+h'} \ \ \text{for} \ h,h' \in P^{\vee}, \\
			& q^h \mathtt a_{jl}q^{-h} = q^{l \langle h, \alpha_j \rangle} \mathtt a_{jl},
			\ \ q^h \mathtt b_{jl}q^{-h} = q^{-l \langle h, \alpha_j \rangle } \mathtt b_{jl}
			\ \ \text{for} \ h \in P^{\vee}\ \text{and}\ (j,l)\in I^{\infty}, \\
			& \sum_{r + s = 1 - l a_{ij}} (-1)^r
			{\mathtt a}_i^{(r)}\mathtt a_{jl}\mathtt a_i^{(s)}=0
			\ \ \text{for} \ i\in 	I^{\text{re}},\ (j,l)\in I^{\infty} \ \text {and} \ i \neq (j,l), \\
			& \sum_{r + s = 1 - l a_{ij}} (-1)^r
			{\mathtt b}_i^{(r)}\mathtt b_{jl}\mathtt b_i^{(s)}=0
			\ \ \text{for} \ i\in 	I^{\text{re}},\ (j,l)\in I^{\infty} \ \text {and} \ i \neq (j,l), \\
			& \mathtt a_{il}\mathtt b_{jk} - \mathtt b_{jk}\mathtt a_{il}=\delta_{ij}\delta_{kl}\tau_{il}(K_{i}^{l} - K_{i}^{-l}), \\
			& \mathtt a_{il}\mathtt a_{jk}-\mathtt a_{jk}\mathtt a_{il} = \mathtt b_{il}\mathtt b_{jk}-\mathtt b_{jk}\mathtt b_{il} =0 \ \ \text{for} \ a_{ij}=0,
		\end{aligned}
	\end{equation}
	where $\tau_{il} = (1 - q_{i}^{2l})^{-1}  \text{ for } (i, l) \in I^{\infty} \text{ and } K_i= q_{i}^{h_{i}} \text{ for } (i \in I)$.
\end{definition}

\vskip 3mm

Let $U^+_q(\g)$ (resp. $U^-_q(\g)$) be the subalgebra of $U_q(\g)$ generated by $\mathtt a_{il}$
(resp. $\mathtt b_{il}$) for $(i,l)\in I^{\infty}$
and let $U^{0}_q(\g)$ be the subalgebra of $U_q(\g)$ generated by $q^h$ $(h\in P^{\vee})$.
Then we have the {\it triangular decomposition} of $U_q(\g)$ \cite{KK20}

$$U_{q}(\g) \cong U_{q}^-(\g) \otimes U_{q}^0(\g) \otimes U_{q}^+(\g).$$

\vskip 2mm

As vector space, the space $U^-_q(\g)$ has the following decomposition
\begin{equation}\label{eq:decomposition U}
U^-_q(\g)=\bigoplus_{\alpha\in P} U^-_q(\g)_{-\alpha},
\end{equation}
where
\begin{equation}\label{eq:weight space U}
U^-_q(\g)_{-\alpha} = \{u \in U^-_q(\g) \mid q^h u q^{-h} = q^{-\langle h, \alpha\rangle} u\ \text{for all}\ h\in P^\vee\}.
\end{equation}

\vskip 2mm

Let $^{-}:U_q(\g)\rightarrow U_q(\g)$ be the $\Q$-linear involution given by
\begin{equation}\label{eq:bar}
	\overline{\mathtt a_{il}}=\mathtt a_{il},\quad \overline{\mathtt b_{il}}=\mathtt b_{il},\quad \overline{K_i}=K_i^{-1},\quad \overline{q}=q^{-1}
\end{equation} 	
for $(i,l)\in I^{\infty}$ and $i\in I$.

\vskip 2mm 

A $U_{q}(\g)$-module $V$ is called a {\it highest weight module with highest weight $\lambda$} if there is a non-zero vector $v_{\lambda}$ in $V$ such that
\begin{enumerate}
	\vskip 2mm
	\item[(i)] $q^{h} \,  v_{\lambda} = q^{\langle h, \lambda \rangle} v_{\lambda}$ for all $h \in P^{\vee}$,
	\vskip 2mm	
	\item[(ii)] $e_{il} \, v_{\lambda} = 0$ for all $(i,l) \in I^{\infty}$,
	\vskip 2mm	
	\item[(iii)] $V=U_{q}(\g) v_{\lambda}$.
\end{enumerate}

\vskip 2mm

The vector $v_{\lambda}$ is called a {\it highest weight vector} with highest weight $\lambda$.
Let $V_{\lambda} = \Q(q) v_{\lambda}$. Then $V$ has a weight space decomposition
$V = \bigoplus_{\mu \le \lambda} V_{\mu}$. For each $\lambda \in P$, there exists a unique irreducible highest weight module,
which is denoted by $V(\lambda)$.

\vskip 2mm

Let $M$ be a $U_{q}(\g)$-module. We say that $M$ has
a {\it weight space decomposition} if
$$M = \bigoplus_{\mu \in P} M_{\mu}, \ \ \text{where}
\ M_{\mu} = \{ m \in M \mid q^{h} \,  m = q^{\langle h, \mu \rangle} m \
\text{for all} \ h \in P^{\vee} \}.$$
\noindent
We denote  $\text{wt}(M):=\{\mu \in
\h^* \mid M_{\mu} \neq 0 \}$.

\vskip 2mm

\begin{definition} \label{def:Oint}
	
	The {\it category} $\mathcal O_{\text{int}}$ consists of $U_{q}(\g)$-modules $M$ such that
	
	\begin{enumerate}
		
		\item[(a)] $M$ has a weight space decomposition $M = \oplus_{\mu \in P} M_{\mu}$
		with $\dim M_{\mu} < \infty$ for all $\mu \in \text{wt}(M)$,
		
		\vskip 2mm
		
		\item[(b)] there exist finitely many wrights $\lambda_{1}, \ldots, \lambda_{s} \in P$ such that
		$$\text{wt}(M) \subset \bigcup_{j=1}^s (\lambda_j - R_{+}),$$
		
		\vskip 2mm
		
		\item[(c)] if $i \in I^{\text{re}}$, ${\mathtt b}_{i}$ is locally nilpotent
		on $M$,
		
		\vskip 2mm
		
		\item[(d)] if $i \in I^{\text{im}}$, we have $\langle h_i, \mu
		\rangle \ge 0$ for all $\mu \in \text{wt}(M)$,
		
		\vskip 2mm
		
		\item[(e)] if $i \in I^{\text{im}}$ and $\langle h_{i}, \mu \rangle =0$,
		then ${\mathtt b}_{il}(M_{\mu})=0$,
		
		\vskip 2mm
		
		\item[(f)] if $i \in I^{\text{im}}$ and $\langle h_i, \mu \rangle
		\le - l a_{ii}$, then ${\mathtt a}_{il}(M_{\mu})=0$.
		
	\end{enumerate}
\end{definition}

\begin{proposition}\cite{KK20} \hfill
	
	{\rm
		\begin{itemize}
			
			\vskip 2mm
			
			\item[(a)] The category ${\mathcal O}_{\text{int}}$ is semisimple.

			\vskip 2mm
			
			\item[(b)] 	If  $\lambda \in P^{+}$, then $V(\lambda)$ is a simple object in $\mathcal O_{\rm int}$.
			
			\vskip 2mm
			
			\item[(c)] Every simple object in the category ${\mathcal O}_{\text{int}}$ has the form
			$V(\lambda)$ for some $\lambda \in P^{+}$.
			
		\end{itemize}
		
	}
	
\end{proposition}

\vskip 2mm

There exists a unique non-degenerate symmetric
bilinear form $(\ , \ )_{K}$ on $V(\lambda)$ given by
\begin{equation} \label{eq:bilinearV}
	\begin{aligned}
		& (v_{\lambda}, v_{\lambda})_{K} = 1, \ \
		(q^{h} u, v)_{K} = (u , q^{h} v)_{K},\\
		& (\mathtt{b}_{il} u, v)_{K} = - (u, K_{i}^{l} \mathtt{a}_{il}v)_{K}, \\
		& (\mathtt{a}_{il} u, v)_{K} = - (u, K_{i}^{-l} \mathtt{b}_{il} v)_{K},
	\end{aligned}
\end{equation}
where $u,v\in V(\lambda)$ and $h \in P^{\vee}$.

\section{Lower crystal bases}
\label{sec:lcb}

Let  $\mathbf{c} = (c_1, \ldots, c_r) \in \Z_{\ge 0}^r$ be a
sequence of non-negative integers.
We define $|\mathbf{c}|:= c_1 + \cdots + c_r$.
If $|\mathbf{c}| = l\geq 0$, then $\mathbf{c}$ is a composition of $l$, denoted  by
$\mathbf{c} \vdash l$. If $|\mathbf{c}| = l\geq 0$ and $c_1 \ge c_2 \ge \ldots \ge c_r$, then $\mathbf{c}$ is a partition of $l$.
For each $i\in I^{\text{im}} \setminus I^{\text{iso}}$ (resp. $i \in I^{\text{iso}}$),
we denote by $\mathcal{C}_{i, l}$ the set of compositions (resp. partitions) of $l$
and set $\mathcal{C}_{i} = \bigsqcup_{l \ge 0} \mathcal{C}_{i,l}$.
For $i \in I^{\text{re}}$, we define $\mathcal{C}_{i,l} = \{ l \}$.

\vskip 3mm

For $\mathbf{c} = (c_1, \ldots, c_r)$, we define
$${\mathtt a}_{i, \mathbf{c}} = {\mathtt a}_{i c_1} {\mathtt a}_{i c_2} \cdots {\mathtt a}_{i c_r},
\quad {\mathtt b}_{i, \mathbf{c}} = {\mathtt b}_{i c_1}{\mathtt b}_{i c_2} \cdots {\mathtt b}_{i c_r}.$$

\noindent
Note that the set $\{ {\mathtt a}_{i, \mathbf{c}} \mid \mathbf{c} \vdash l \}$
(resp. $\{ {\mathtt b}_{i, \mathbf{c}} \mid \mathbf{c} \vdash l \}$)
is a basis of $U_{q}(\g)_{l \alpha_i}$ (resp. $U_{q}(\g)_{- l \alpha_i}$).

\vskip 5mm

\subsection{Crystal bases for $V(\lambda)$} \label{sub:Vlambda} \hfill

\vskip 2mm

Let $M = \oplus_{\mu \in P} M_{\mu}$ be a $U_{q}(\g)$-module in the category ${\mathcal O}_{\text{int}}$
and let  $u \in M_{\mu}$ for $\mu \in \text{wt}(M)$.

\vskip 3mm

For $i \in I^{\text{re}}$, by \cite{Kashi91}, the vector $u$ can be written uniquely as
\begin{equation} \label{eq:real string}
	u = \sum_{k \ge 0} {\mathtt b}_{i}^{(k)} u_k
\end{equation}
such that

\vskip  2mm

\begin{enumerate}
	
	\item[(i)] ${\mathtt a}_{i} u_k = 0$ for all $k \ge 0$,
	
	\vskip 2mm
	
	\item[(ii)] $u_k \in M_{\mu + k \alpha_{i}}$,
	
	\vskip 2mm
	
	\item[(iii)] $u_{k} = 0$ if $\langle h_{i}, \mu + k \alpha_{i} \rangle = 0 $.
	
\end{enumerate}

\vskip 3mm

For $i \in I^{\text{im}}$, by \cite{Bozec2014b, Bozec2014c}, the vector
$u$ can be written uniquely as
\begin{equation} \label{eq:imaginary string}
	u = \sum_{\mathbf c\in\mathcal C_i} {\mathtt b}_{i, \mathbf{c}} u_{\mathbf{c}}
\end{equation}
such that

\vskip 2mm

\begin{enumerate}
	
	\item[(i)] ${\mathtt a}_{ik} u_{\mathbf{c}} = 0$ for all $k >0$,
	
	\vskip 2mm
	
	\item[(ii)] $u_{\mathbf{c}} \in M_{\mu+ |\mathbf{c}| \alpha_{i}}$,
	
	\vskip 2mm
	
	\item[(iii)] $u_{\mathbf{c}} = 0$ if $\langle h_{i}, \mu + |\mathbf{c}| \alpha_{i} \rangle = 0$.
	
\end{enumerate}

\vskip 2mm

The expressions \eqref{eq:real string}, \eqref{eq:imaginary string} are called the
{\it $i$-string decomposition} of $u$.

\vskip 3mm

Given the $i$-string decompositions \eqref{eq:real string}, \eqref{eq:imaginary string},
we define the {\it lower Kashiwara operators} on $M$ as follows.

\vskip 3mm

\begin{definition} \label{def:Kashiwara operator} \hfill
	
	\vskip 3mm
	
	(a) For $i \in I^{\text{re}}$, we define
	
	\begin{equation*} 
		\begin{aligned}
			\widetilde{e}_{i} u = \sum_{k \ge 1} {\mathtt b}_{i}^{(k-1)} u_{k}, \quad
		 \widetilde{f}_{i}  u = \sum_{k \ge 0} {\mathtt b}_{i}^{(k+1)} u_{k}.
		\end{aligned}
	\end{equation*}
	
	\vskip 2mm
	
	(b) For $i \in I^{\text{im}}\setminus I^{\text{iso}}$ and $l>0$, we define
	
	\begin{equation*} 
		\begin{aligned}
		\widetilde{e}_{il} u = \sum_{\mathbf{c} \in {\mathcal C}_{i}:  c_1 = l} \,
			{\mathtt b}_{i, \mathbf{c} \setminus c_1} u_{\mathbf{c}},\quad
			 \widetilde{f}_{il}  u = \sum_{\mathbf{c} \in {\mathcal C}_{i}}
			{\mathtt b}_{i,(l, \mathbf{c})} u_{\mathbf{c}}.
		\end{aligned}
	\end{equation*}
	
	\vskip 2mm
	
	(c) For $i \in I^{\text{iso}}$ and $l>0$, we define
	
	\begin{equation*}
		\begin{aligned}
 \widetilde{e}_{il} u = \sum_{\mathbf{c} \in {\mathcal C}_{i}} \,  \mathbf{c}_{l} \,
			{\mathtt b}_{i, \mathbf{c} \setminus \{l\}} u_{\mathbf{c}},\quad
 \widetilde{f}_{il}  u = \sum_{\mathbf{c} \in {\mathcal C}_{i}} \, \frac{1} {\mathbf{c}_{l} + 1} \,
			{\mathtt b}_{i,\{l\} \cup \mathbf{c})} u_{\mathbf{c}},
		\end{aligned}
	\end{equation*}
	where ${\mathbf c}_{l}$ denotes the number of $l$ in $\mathbf{c}$.
	
\end{definition}

\vskip 2mm

It is easy to see that $\widetilde{e}_{il}\circ\widetilde{f}_{il}=\text{id}_{M_{\mu}}$ for
$(i,l) \in I^{\infty}$ and $\langle h_i,\mu \rangle>0$.

\vskip 3mm

Let $\A_{0} = \{f \in \Q(q) \mid \text{$f$ is regular at $q=0$} \}$.
Then we have an isomorphism

$$\A_{0} / q \A_{0} \cong \Q, \quad f + q \A_{0} \longmapsto f(0).$$

\vskip 3mm

\begin{definition} \label{def:crystal lattice Llambda} 
	
	Let $M$ be a $U_{q}(\g)$-module in the category ${\mathcal O}_{\text{int}}$
	and let $L$ be a free $\A_{0}$-submodule of $M$.
	The submodule $L$ is called a {\it crystal lattice} of $M$ if
	
	\begin{enumerate}
		
		\vskip 2mm
		
		\item[(a)] $\Q \otimes_{\A_{0}} L \cong M$,
		
		\vskip 2mm
		
		\item[(b)] $L =\oplus_{\mu \in P} L_{\mu}$, where $L_{\mu} = L \cap M_{\mu}$,
		
		\vskip 2mm
		
		\item[(c)] $\widetilde{e}_{il} L \subset L$, $\widetilde{f}_{il} L \subset L$
		for $(i, l) \in I^{\infty}$.
		
	\end{enumerate}
	
\end{definition}

\vskip 3mm

Since the operators $\widetilde e_{il}$, $\widetilde f_{il}$ preserve $L$, by abuse of notations,
they induce the operators

$$\widetilde{e}_{il}, \, \widetilde{f}_{il}: L/qL \longrightarrow L/qL.$$

\vskip 3mm

\begin{definition} \label{def:crystal basis Blambda} 
	
	Let $M$ be a $U_{q}(\g)$-module in the category ${\mathcal O}_{\text{int}}$.
	A {\it crystal basis} of $M$ is a pair $(L, B)$ such that
	\begin{enumerate}
		
		\vskip 2mm
		
		\item[(a)] $L$ is a crystal lattice of $M$,
		
		\vskip 2mm
		
		\item[(b)] $B$ is a $\Q$-basis of $L/qL$,
		
		\vskip 2mm
		
		\item[(c)] $B=\sqcup_{\mu \in P}\ B_\mu$, where
		$B_\mu=B\cap{(L/qL)}_\mu$,
		
		\vskip 2mm
		
		\item[(d)] $\widetilde{e}_{il}B\subset B\cup\{0\}$, $\widetilde{f}_{il} B\subset B\cup\{0\}$
		for $(i, l) \in I^{\infty}$,
		
		\vskip 2mm
		
		\item[(e)] for any $b,b'\in B$ and $(i,l)\in I^{\infty}$, we have
		$\widetilde{f}_{il}b=b'$ if and only if $b=\widetilde{e}_{il}b'$.
		
	\end{enumerate}
\end{definition}

For $(i,l)\in I^\infty$, we set $E_{il}:=-K_i^l\mathtt a_{il}$. Then by \eqref{eq:primitive}, we have

\begin{align*}
	E_{il}\mathtt b_{jk}-q_i^{-kla_{ij}}\mathtt b_{jk}E_{il}=\delta_{ij}\delta_{kl}\frac{1-K_i^{2l}}{1-q_i^{2l}}.\label{Eb}
\end{align*}

\vskip 2mm

\begin{lemma}\cite{FHKK}
	\label{euEuqL}
	{\rm
		Let $M$ be a $U_q(\g)$-module in the category ${\mathcal O}_{\text{\rm int}}$
		and $(L, B)$ be a crystal basis of $M$.
		For any $u\in M_{\mu}$, we have
		
		\begin{equation*}
			\widetilde{e}_{il}\,u \equiv E_{il}\,u \ \text{mod} \ qL
			\ \ \text{for} \ (i,l) \in I^{\infty}.
		\end{equation*}
	}
\end{lemma}

\vskip 2mm

Let $V(\lambda)=U_q(\mathfrak g)v_\lambda$ be the irreducible highest weight $U_q(\g)$-module
with highest weight $\lambda \in P^+$.
Let $L(\lambda)$ be the free $\A_0$-submodule of $V(\lambda)$
spanned by $\widetilde{f}_{i_1l_1}\cdots\widetilde{f}_{i_rl_r} v_\lambda$ $(r\geq 0, (i_k,l_k)\in I^{\infty})$
and let

\begin{align*}
	B(\lambda):=\{\widetilde{f}_{i_1l_1}\cdots\widetilde{f}_{i_rl_r}v_\lambda+qL(\lambda)\}\setminus\{0\}.
\end{align*}

\vskip 3mm

\begin{theorem}\cite{FHKK}	
	\label{thm:crystal basis of V}	
	{\rm
		The pair $(L(\lambda),B(\lambda))$ is a crystal basis of $V(\lambda)$.
	}
\end{theorem}

\vskip 3mm

\subsection{Crystal bases for $U_{q}^{-}(\g)$} \label{sub:Uminus} \hfill
\vskip 2mm

Let $(i, l) \in I^{\infty}$ and $S\in U_{q}^-(\g)$.
Then there exist unique elements $T, W \in U_{q}^-(\g)$ such that
\begin{equation*}
	\mathtt a_{il}S-S\mathtt a_{il}=\frac{K_i^{l}T-K_i^{-l}W}{1-q_i^{2l}}.
\end{equation*}

\vskip 2mm

Set $A_{il}:=-q^l_i\mathtt a_{il}$. Then we have

\begin{equation}\label{eq:AP-PA}
	A_{il}S-SA_{il}=\frac{K_i^lT-K_i^{-l}W}{q_i^l-q_i^{-l}}.
\end{equation}

\vskip 3mm

We  define the operators $e'_{il}, e''_{il}:U_{q}^-(\g) \longrightarrow U_{q}^-(\g)$ by

\begin{equation}\label{eq:eP=R}
	e'_{il}(S)=W,\quad e''_{il}(S)=T.
\end{equation}

\vskip 3mm

\noindent
By \eqref{eq:AP-PA} and \eqref{eq:eP=R}, we have

\begin{equation}\label{eq:AP-PA=}
	A_{il}S-SA_{il}=\frac{K_i^l(e''_{il}(S))-K_i^{-l}(e'_{il}(S))}{q_i^l-q_i^{-l}}.
\end{equation}

\vskip 3mm

Therefore we obtain

\begin{equation*}
	\begin{aligned}
		&e'_{il}\mathtt b_{jk}=\delta_{ij}\delta_{kl}+q_i^{-kla_{ij}}\mathtt b_{jk}e'_{il},\\
		&e''_{il}\mathtt b_{jk}=\delta_{ij}\delta_{kl}+q_i^{kla_{ij}}\mathtt b_{jk}e''_{il},\\
		&e'_{il}e''_{jk}=q_i^{kla_{ij}}e''_{jk}e'_{il}.
	\end{aligned}
\end{equation*}

\vskip 4mm

Let $u \in U_{q}^-(\g)_{-\alpha}$ with $\alpha \in R_{+}$.
For $i \in I^{\text{re}}$, by \cite{Kashi91}, the vector $u$ can be written uniquely as
\begin{equation} \label{eq:real string U}
	u = \sum_{k \ge 0} {\mathtt b}_{i}^{(k)} u_k
\end{equation}
such that

\vskip  2mm

\begin{enumerate}
	
	\item[(i)] $e_{i}'u_k = 0$ for all $k \ge 0$,
	
	\vskip 2mm
	
	\item[(ii)] $u_k \in U_{q}^-(\g)_{-\alpha + k \alpha_{i}}$,
	
	\vskip 2mm
	
	\item[(iii)] $u_{k} = 0$ if $\langle h_{i}, -\alpha + k \alpha_{i} \rangle = 0 $.
	
\end{enumerate}

\vskip 3mm

For $i \in I^{\text{im}}$, by \cite{Bozec2014b, Bozec2014c}, the vector
$u$ can be written uniquely as
\begin{equation} \label{eq:imaginary string U}
	u = \sum_{\mathbf c\in \mathcal C_i} {\mathtt b}_{i, \mathbf{c}} u_{\mathbf{c}}
\end{equation}
such that

\vskip 2mm

\begin{enumerate}
	
	\item[(i)] $e_{ik}' u_{\mathbf{c}} = 0$ for all $k >0$,
	
	\vskip 2mm
	
	\item[(ii)] $u_{\mathbf{c}} \in U_{q}^-(\g)_{- \alpha+ |\mathbf{c}| \alpha_{i}}$,
	
	\vskip 2mm
	
	\item[(iii)] $u_{\mathbf{c}} = 0$ if $\langle h_{i}, - \alpha+ |\mathbf{c}| \alpha_{i} \rangle = 0$.
	
\end{enumerate}

\vskip 2mm

The expressions \eqref{eq:real string U}, \eqref{eq:imaginary string U} are called the
{\it $i$-string decomposition} of $u$.

%Note that (i) is equivalent to saying that
%$A_{i,k} u_{\mathbf{c}} = E_{i,k} u_{\mathbf{c}} = 0$ for all  $k>0$.

%\begin{remark}
%For the $i$-string decomposition of $u$,  set
%\begin{align*}
%&u':=\sum_{\mathbf c\in\mathcal C_i}\mathtt b_{i,\mathbf c\setminus\mathbf c'}u_{\mathbf c}	\ \text{for some } \mathbf c'\neq\mathbf 0, \\
%& u'':=\sum_{\mathbf c\in\mathcal C_i}\mathtt b_{i,(\mathbf c,\mathbf c')}u_{\mathbf c}	\ \text{for some } \mathbf c'\neq\mathbf 0.
%\end{align*}

%Then the vector $u'$ is an $i$-string decomposition,
%but the vector $u''$ is not necessarily an $i$-string decomposition.
%Note that $(\mathbf c,\mathbf c')=\mathbf c\cup\mathbf c'$ and $\mathbf c\setminus\mathbf c'=\mathbf c\setminus\{\mathbf c'\}$ for $i\in I^{\text{iso}}$.

%\end{remark}

\vskip 3mm

Given the $i$-string decompositions \eqref{eq:real string U}, \eqref{eq:imaginary string U},
we define the {\it lower Kashiwara operators} on $U_{q}^-(\g)$ as follows.

\vskip 3mm

\begin{definition} \label{def:Kashiwara operators} \hfill
	
	\vskip 2mm
	
	(a) For $i \in I^{\text{re}}$, we define
	
	\begin{equation*} 
		\begin{aligned}
	 \widetilde{e}_{i} u = \sum_{k \ge 1} {\mathtt b}_{i}^{(k-1)} u_{k}, \quad
	 \widetilde{f}_{i}  u = \sum_{k \ge 0} {\mathtt b}_{i}^{(k+1)} u_{k}.
		\end{aligned}
	\end{equation*}
	
	\vskip 2mm
	
	(b) For $i \in I^{\text{im}}\setminus I^{\text{iso}}$ and $l>0$, we define
	
	\begin{equation*} 
		\begin{aligned}
		 \widetilde{e}_{il} u = \sum_{\mathbf{c} \in {\mathcal C}_{i}:  c_1 = l} \,
			{\mathtt b}_{i, \mathbf{c} \setminus c_1} u_{\mathbf{c}},\quad
		 \widetilde{f}_{il}  u = \sum_{\mathbf{c} \in {\mathcal C}_{i}}
			{\mathtt b}_{i,(l, \mathbf{c})} u_{\mathbf{c}}.
		\end{aligned}
	\end{equation*}
	
	\vskip 2mm
	
	(c) For $i \in I^{\text{iso}}$ and $l>0$, we define
	
	\begin{equation*} 
		\begin{aligned}
	\widetilde{e}_{il} u = \sum_{\mathbf{c} \in {\mathcal C}_{i}} \, \mathbf{c}_{l} \,
			{\mathtt b}_{i, \mathbf{c} \setminus \{l\}} u_{\mathbf{c}},\quad
	 \widetilde{f}_{il}  u = \sum_{\mathbf{c} \in {\mathcal C}_{i}} \, \frac{1}{ \mathbf{c}_{l} + 1} \,
			{\mathtt b}_{i,\{l\} \cup \mathbf{c})} u_{\mathbf{c}},
		\end{aligned}
	\end{equation*}
	where ${\mathbf c}_{l}$ denotes the number of $l$ in $\mathbf{c}$.
	
\end{definition}

\vskip 2mm

It is easy to see that $\widetilde{e}_{il}\circ\widetilde{f}_{il}=\text{id}_{U_{q}^-(\g)_{-\alpha}}$ for
$(i,l) \in I^{\infty}$ and $\langle h_i, - \alpha \rangle>0$.

\vskip 3mm

\begin{definition} \label{def:crystal lattice U}
	
	\vskip 2mm
	
	A free $\A_0$-submodule $ L$ of $U_{q}^-(\g)$ is called a {\it crystal lattice} if the following conditions hold.
	
	\vskip 2mm
	
	\begin{enumerate}
		
		\vskip 2mm
		
		\item[{\rm (a)}] $\Q(q)\otimes_{\mathbf A_0} L\cong U_q^-(\g)$,
		
		\vskip 2mm
		
		\item[{\rm (b)}] $ L=\oplus_{\alpha\in R_+} L_{-\alpha}$, where $ L_{-\alpha}= L\cap {U_{q}^-(\g)}_{-\alpha}$,
		
		\vskip 2mm
		
		\item[{\rm (c)}] $\widetilde e_{il} L\subset  L$, \ $\widetilde f_{il} L\subset  L$ for all $(i,l)\in I^{\infty}$.
		
	\end{enumerate}
	
\end{definition}

\vskip 3mm

The condition (c) yields  the $\Q$-linear maps

$${\widetilde e}_{il}, \, {\widetilde f}_{il}: L/qL \longrightarrow L/qL.$$

\vskip 3mm

\begin{definition}
	
	A {\it crystal basis} of $U_{q}^-(\g)$ is a pair $( L, B)$ such that
	
	\vskip 2mm
	
	\begin{enumerate}
		
		\vskip 2mm
		
		\item[{\rm (a)}] $L$ is a crystal lattice of $U_{q}^-(\g)$,
		
		\vskip 2mm
		
		\item[{\rm (b)}] $B$ is a $\Q$-basis of $L/q L$,
		
		\vskip 2mm
		
		\item[{\rm (c)}] $ B=\sqcup_{\alpha \in R_+} B_{-\alpha}$, where
		$ B_{-\alpha}= B\cap{( L/q L)}_{-\alpha}$,
		
		\vskip 2mm
		
		\item[{\rm (d)}] $\widetilde{e}_{il} B\subset B\cup\{0\}$, \ $\widetilde{f}_{il} B \subset B\cup\{0\}$ \,
		for $(i,l) \in I^{\infty}$,
		
		\vskip 2mm
		
		\item[{\rm (e)}] for any $b,b'\in B$ and $(i,l) \in I^{\infty}$,
		we have $\widetilde{f}_{il}b=b'$ if and only if $b=\widetilde{e}_{il}b'$.
		
	\end{enumerate}
	
\end{definition}

\vskip 2mm

Let $L(\infty)$ be the $\A_0$-submodule of $U_{q}^-(\g)$ spanned by $\widetilde{f}_{i_1l_1}\cdots\widetilde{f}_{i_rl_r}\mathbf 1$ $(r\geq 0, (i_j,l_j)\in I_{\infty})$,
and $B(\infty)=\{\widetilde{f}_{i_1l_1}\cdots\widetilde{f}_{i_rl_r}\mathbf 1+qL(\infty)\}$.

\vskip 3mm

\begin{theorem}\cite{FHKK}
	\label{thm:crystal basis of U}
	{\rm
		The pair  $(L(\infty),B(\infty))$ is a crystal basis of  $U^-_q(\mathfrak g)$.
	}
\end{theorem}

By extracting the fundamental properties of the crystal bases of $V(\lambda)$
and $U_{q}^-(\g)$, we define the notion of abstract crystals as follows.

\vskip 3mm

\begin{definition}\label{def: abstract crystal} 
	
	An {\it abstract crystal}  is a set $B$ together with the maps ${\rm wt}\colon  B \rightarrow P$,
	$\varphi_i,\varepsilon_i\colon  B\rightarrow \Z\cup \{-\infty\}$ $(i\in I)$ and $\widetilde{e}_{il},\widetilde{f}_{il}\colon  B\rightarrow   B\cup \{0\}$ $((i,l)\in I^\infty)$ satisfying the following conditions:
	\vskip 2mm
	\begin{enumerate}
		\item[{\rm (a)}] $\text{wt}(\widetilde{f}_{il}b)=\text{wt}(b)-l\alpha_i$ if $\widetilde{f}_{il}b\neq 0$,\quad $\text{wt}(\widetilde{e}_{il}b)=\text{wt}(b)+l\alpha_i$ if $\widetilde{e}_{il}b\neq 0$.
		
		\vskip 2mm
		
		\item[{\rm (b)}] $\varphi_i(b)= \langle h_i, \text{wt}(b) \rangle+\varepsilon_i(b)$ for $i\in I$ and $b\in  B$.
		
		\vskip 2mm
		
		\item[{\rm (c)}] $\widetilde{f}_{il}b=b'$ if and only if $b=\widetilde{e}_{il}b'$ for $(i,l)\in I^\infty$ and $b,b'\in  B$.
		
		\vskip 2mm
		
		\item[{\rm (d)}]  For any $i\in I^{\text{re}}$ and $b\in  B$, we have
		\vskip 2mm
		\begin{itemize}
			\item [{\rm (i)}] $\varepsilon_i(\widetilde{f}_ib)=\varepsilon_i(b)+1$, $\varphi_i(\widetilde{f}_ib)=\varphi_i(b)-1$ if $\widetilde{f}_ib\neq 0$,
			\vskip 2mm
			\item [{\rm (ii)}] $\varepsilon_i(\widetilde{e}_ib)=\varepsilon_i(b)-1$, $\varphi_i(\widetilde{e}_ib)=\varphi_i(b)+1$ if $\widetilde{e}_ib\neq 0$.
		\end{itemize}
		
		\vskip 2mm
		
		\item[{\rm (e)}] For any $i\in I^{\text{im}}$, $l>0$ and $b\in B$, we have
		\vskip 2mm
		\begin{itemize}
			\item [{\rm (i$'$)}] $\varepsilon_i(\widetilde{f}_{il}b)=\varepsilon_i(b)$, $\varphi_i(\widetilde{f}_{il}b)=\varphi_i(b)-la_{ii}$ if $\widetilde{f}_{il}b\neq 0$,
			\vskip 2mm
			\item [{\rm (ii$'$)}] $\varepsilon_i(\widetilde{e}_{il}b)=\varepsilon_i(b)$, $\varphi_i(\widetilde{e}_{il}b)=\varphi_i(b)+la_{ii}$ if $\widetilde{e}_{il}b\neq 0$.
		\end{itemize}
		
		\vskip 2mm
		
		\item[{\rm (f)}]  For any $(i,l)\in I^\infty$ and $b\in B$ such that $\varphi_i(b)=-\infty$, we have $\widetilde{e}_{il}b=\widetilde{f}_{il}b=0$.
	\end{enumerate}
\end{definition}

	\vskip 2mm

\begin{definition} \hfill
	
	\vskip 2mm
	
	\begin{enumerate}
		\item[{\rm (a)}] A {\it crystal morphism} $\psi$ between two abstract crystals $ B_1$ and $ B_2$ is a map from $ B_1$ to
		$ B_2\sqcup\{0\}$ satisfying the following conditions:
		
		\vskip 2mm
		
		\begin{enumerate}
			\item[{\rm (i)}] for $b\in  B_1$ and $i\in I$, we have $\text{wt}(\psi(b))=\text{wt}(b)$,  $\varepsilon_i(\psi(b))=\varepsilon_i(b)$, $\varphi_i(\psi(b))=\varphi_i(b)$,
			\item[{\rm (ii)}] for $b\in B_1$ and $(i,l)\in I^\infty$ satisfying $\widetilde{f}_{il}b\in B_1$, we have $\psi(\widetilde{f}_{il}b)=\widetilde{f}_{il}\psi(b)$.
		\end{enumerate}	
		
		\vskip 2mm
		
		\item[{\rm (b)}] A crystal morphism $\psi: B_1\to  B_2$ is called {\it strict} if
		\begin{equation*}
			\psi(\widetilde{e}_{il} b) = \widetilde{e}_{il}(\psi(b)), \quad \psi(\widetilde{f}_{il} b) = \widetilde{f}_{il}(\psi(b))	
		\end{equation*}	
		for all $(i,l)\in I^\infty$ and $b \in  B_1$.	
	\end{enumerate}	
\end{definition}

\section{Lower global bases}
\label{sec:lgb}
Let $\A=\Z[q,q^{-1}]$, $\A_{\Q} = \Q[q, q^{-1}]$ and $\A_{\infty}$
be the subring of $\Q(q)$ consisting of rational functions which are regular at $q=\infty$.

\vskip3mm

\begin{definition} \label{def:balanced triple}
	
	Let $V$ be a  $\Q(q)$-vector space.
	Let
	$V_{\Q}$, $L_{0}$ and $L_{\infty}$ be an $\A_{\Q}$-lattice, $\A_{0}$-lattice and
	$\A_{\infty}$-lattice, respectively.
	We say that $(V_{\Q}, L_{0}, L_{\infty})$ is a
	{\it balanced triple} for $V$ if the following conditions hold:
	
	\vskip 2mm
	
	\begin{enumerate}
		
		\item[(a)] The $\Q$-vector space $V_{\Q}\cap L_0\cap L_\infty$ is a free $\Q$-lattice of the
		$\A_0$-module $ L_0$.
		
		\vskip 2mm
		
		\item[(b)] The $\Q$-vector space $V_{\Q}\cap L_0\cap L_\infty$ is a free $\Q$-lattice of the
		$\A_\infty$-module $ L_\infty$.
		\vskip 2mm
		
		\item [(c)]The $\Q$-vector space $V_{\Q}\cap L_0\cap L_\infty$ is a free $\Q$-lattice of the
		$\A_{\Q}$-module $V_{\Q}$.
	\end{enumerate}
\end{definition}

\vskip 3mm

\begin{theorem}\cite{HK02, Kashi91} \label{thm:global basis}  %[Theorem 6.1.4]
	{\rm
		The following statements are equivalent.
		
		\vskip 2mm
		
		\begin{enumerate}
			\item[{\rm (a)}] $(V_{\Q}, L_0, L_\infty)$ is  a balanced triple.
			
			\vskip 2mm
			
			\item[{\rm (b)}] The canonical map $V_{\Q}\cap L_0\cap L_\infty\to L_0/q L_0$ is an isomorphism.
			
			\vskip 2mm
			
			\item[{\rm (c)}] The canonical map $V_{\Q}\cap L_0\cap L_\infty\to L_\infty/q L_\infty$ is an isomorphism.
			
		\end{enumerate}
		
	}
	
\end{theorem}

\vskip 3mm

Let $(V_{\Q}, L_0, L_\infty)$ be a balanced triple and let
\begin{equation*}
	G: L_0/q L_0\longrightarrow V_{\Q}\cap L_0\cap L_\infty
\end{equation*}
be the inverse of the canonical isomorphism
$V_{\Q}\cap L_0\cap L_\infty\stackrel{\sim}{\longrightarrow} L_0/q L_0$.

\vskip 3mm

\begin{proposition} \cite{Kashi91, HK02} 	
	{\rm
		If $B$ is a $\Q$-basis of $ L_0/q L_0$,
		then $\B :=\{G(b)\mid b\in B\}$ is an $\A_{\Q}$-basis of $V_{\Q}$.
	}
\end{proposition}

\vskip 3mm

\begin{definition}
	Let $(V_{\Q}, L_0, L_\infty)$ be a balanced triple for a $\Q(q)$-vector space $V$.
	\begin{enumerate}
		\vskip 2mm
		\item[(a)] A $\Q$-basis $B$ of $ L_0/q L_0$ is called a {\it local basis} of $V$ at $q=0$.
		
		\vskip 2mm
		
		\item[(b)] The $\A_{\Q}$-basis $\B = \{G(b) \mid b\in B \}$ is called the {\it lower global basis} of $V$ corresponding to the local basis $B$.
	\end{enumerate}
\end{definition}	

\vskip 3mm
We define $U^{-}_{\Z}(\g)$ (resp. $U^{-}_{\Q}(\g)$) to be the $\A$-subalgebra
(resp. $\A_{\Q}$-subalgebra) of $U_{q}^{-}(\g)$ generated by
$\mathtt{b}_{i}^{(n)}$ $(i \in I^{\text{re}}, n \ge 0)$ and
${\mathtt b}_{il}$ $(i \in I^{\text{im}}, l>0)$.

\vskip 3mm

Let $V(\lambda) = U_{q}(\g) v_{\lambda}$ be the irreducible highest weight module
with highest weight $\lambda \in P^{+}$.
We  define
$V(\lambda)_{\Z} = U_{\Z}^{-}(\g) \, v_{\lambda}$
and $V(\lambda)_{\Q} = U_{\Q}^{-}(\g) \, v_{\lambda}$.

\begin{theorem}\cite{FHKK}\label{thm:balanced triple}
	{\rm
		There exist $\Q$-linear canonical isomorphisms
		
		\vskip 2mm
		
		\begin{enumerate}
			\item[{\rm (a)}] $U^-_{\Q}(\g)\cap L(\infty)\cap\overline{L(\infty)}\stackrel{\sim}{\longrightarrow} L(\infty)/qL(\infty)$,
			where $^{-} : U_{q}(\g) \rightarrow U_{q}(\g)$ is the $\Q$-linear bar involution defined by \eqref{eq:bar},
			
			\vskip 2mm
			
			\item[{\rm (b)}] ${V(\lambda)_{\Q}}\cap L(\lambda)\cap\overline{L(\lambda)}\stackrel{\sim}{\longrightarrow} L(\lambda)/qL(\lambda)$,
			where $^{-}$ is the $\Q$-linear automorphism on $V(\lambda)$ defined by
			$$
			P\, v_\lambda\mapsto \overline{P} \, v_\lambda \ \ \text{for}\ P\in U_q^-(\g).
			$$
		\end{enumerate}
		
	}
\end{theorem}

\vskip 2mm

Thus we obtain:

\vskip 2mm

\begin{proposition}
	
	{\rm
		Let $G$ denote the inverse of the above isomorphisms.
		
		\begin{enumerate}
			
			\vskip 2mm
			
			\item[(a)] $\B(\infty) : = \{ G(b) \mid b \in B(\infty) \}$ is a
			lower global basis of $U_{\Q}^{-}(\g)$.
			
			\vskip 2mm
			
			\item[(b)] $\B(\lambda) : = \{ G(b) \mid b \in B(\lambda) \}$ is a
			lower global basis of $V(\lambda)_{\Q}$.

		\end{enumerate}
		
	}
	
\end{proposition}

\vskip 2mm

In the next proposition, we will describe the action of Kashiwara
operators on lower global basis.

\vskip 3mm

\begin{proposition}\label{Eil and bil}
	{\rm
		Let $i\in I^{\text{\rm im}}$, $l>0$ and let
		$u=\sum_{\mathbf c\in\mathcal C_i}\mathtt b_{i,\mathbf c}u_{\mathbf c}\in L(\lambda)$
		be the $i$-string decomposition of $u$.
		Set
		$b=u+qL(\lambda)$ and
		$b_{\mathbf c}=\mathtt b_{i,\mathbf c}u_{\mathbf c}+qL(\lambda)$.
		
		\vskip 2mm
		
		Then we have
		
\vskip 2mm		
		\begin{enumerate}
			\item[{\rm (a)}]
			$E_{il}G(b)=G(\widetilde{e}_{il}b)$.
			
			\vskip 2mm
			
			\item[{\rm (b)}]
			$
			\mathtt b_{il}G(b)=
			\begin{cases}
				G(\widetilde{f}_{il}b) &\text{\rm if}\ i\notin I^{\text{\rm iso}},\\
				\sum_{\mathbf c\in\mathcal C_i}(\mathbf c_l+1)G(\widetilde{f}_{il}b_{\mathbf c})
				&\text{\rm if}\ i\in I^{\text{\rm iso}}.
			\end{cases}
			$
		\end{enumerate}
		
	}
\end{proposition}

\begin{proof}
	
	(a) By Lemma \ref{euEuqL}, we have
	$$E_{il}u\equiv \widetilde{e}_{il}u\!\!\!\! \mod{qL(\lambda)},$$
	which implies that
	$$E_{il}G(b)\equiv \widetilde{e}_{il}b\!\!\!\!  \mod{qL(\lambda)}.$$
	
	\vskip 2mm
	
	On the other hand, by the definition of $\B(\lambda)$, we have
	
	$$G(\widetilde{e}_{il}b)\equiv\widetilde{e}_{il}b\!\!\!\!  \mod{qL(\lambda)}.$$
	Hence $E_{il}G(b)=G(\widetilde{e}_{il}b)$.
	
	\vskip 2mm
	
	(b) If $i\notin I^{\text{iso}}$, by Definition \ref{def:Kashiwara operator}, we have
	$\widetilde{f}_{il}u=\mathtt b_{il}u$.
	It follows that
	$$\mathtt b_{il}G(b)\equiv \widetilde{f}_{il}b\!\!\!\!  \mod{qL(\lambda)}.$$
	
	\vskip 2mm
	
	Likewise, by the definition of $\B(\lambda)$, we have
	$$G(\widetilde{f}_{il}b)\equiv\widetilde{f}_{il}b\!\!\!\!  \mod{qL(\lambda)}.$$
	Hence $\mathtt b_{il}G(b)=G(\widetilde{f}_{il}b)$.
	
	\vskip 2mm
	
	If $i\in I^{\text{iso}}$, we first assume that
	$$
	u=\mathtt b_{i,\mathbf c}u_{\mathbf c} \ \text{and}\ E_{ik}u_{\mathbf c}=0\ \text{for all} \ k>0.
	$$
	By Definition \ref{def:Kashiwara operator}, we have
	$$
	\widetilde{f}_{il}u=\frac{1}{\mathbf c_l+1}\mathtt b_{il}(\mathtt b_{i,\mathbf c}u_{\mathbf c}).
	$$
	Hence $\mathtt b_{il}G(b)=\mathtt b_{il}u=(\mathbf c_{l}+1)\widetilde{f}_{il}u$.
	Since
	$\widetilde{f}_{il}b=\widetilde{f}_{il}u+qL(\lambda)$, we obtain $G(\widetilde{f}_{il}b)=\widetilde{f}_{il}u$. It follows that
	$$
	\mathtt b_{il}G(b)=(\mathbf c_l+1)G(\widetilde{f}_{il}b).
	$$
	Thus we have
	$$
	\mathtt b_{il}G(b_{\mathbf c})=(\mathbf c_{l}+1)G(\widetilde{f}_{il}b_{\mathbf c}).
	$$
	
	\vskip 2mm
	
	Therefore $\mathtt b_{il}G(b)=\sum_{\mathbf c\in\mathcal C_i}(\mathbf c_l+1)G(\widetilde{f}_{il}b_{\mathbf c})$.
\end{proof}

\vskip 3mm

Similarly, we have
\begin{corollary}\label{cor:bilGb}
	
	{\rm
		Let $i\in I^{\text{\rm im}}$, $l>0$ and
		let $u=\sum_{\mathbf c\in\mathcal C_i}\mathtt b_{i,\mathbf c}u_{\mathbf c}\in L(\infty)$
		be the $i$-string decomposition of $u$.
		Set
		$b=u+qL(\infty)$ and $b_{\mathbf c}=\mathtt b_{i,\mathbf c}u_{\mathbf c}+qL(\infty)$.
		
		\vskip 2mm
		
		Then we have
		\begin{enumerate}
			
			\vskip 2mm
			
			\item[{\rm (a)}]
			$E_{il}G(b)=G(\widetilde{e}_{il}b)$,
			
			\vskip 2mm
			
			\item[{\rm (b)}]
			$
			\mathtt b_{il}G(b)=
			\begin{cases}
				G(\widetilde{f}_{il}b) &\text{if}\ i\notin I^{\text{iso}},\\
				\sum_{\mathbf c\in\mathcal C_i}(\mathbf c_l+1)G(\widetilde{f}_{il} b_{\mathbf c}) &\text{if}\ i\in I^{\text{iso}}.
			\end{cases}
			$
		\end{enumerate}
		
	}
\end{corollary}

\vskip 2mm

\section{Lower perfect bases} \label{sec:lpb} 

Let $(A, P, P^{\vee}, \Pi, \Pi^{\vee})$ be a Borcherds-Cartan datum.
Let  $V = \bigoplus_{\mu \in P} V_{\mu}$ be a $P$-graded vector space
over a field $\mathbf{k}$ and set
$\text{wt}(V):= \{ \mu \in P \mid V_{\mu} \neq 0 \}$.

\vskip 3mm

\begin{definition} \label{def:weak lower perfect}
	
	{\rm
		Let $\{\mathtt{f}_{il}\}_{(i,l)\in I^{\infty}}$ be a family of endomorphisms of $V$.
		We say that $(V, \{\mathtt{f}_{il}\}_{(i,l)\in I^{\infty}})$ is a {\it weak lower perfect space}
		if
			\vskip 2mm		
		
		\begin{enumerate}
			
			\item[(i)] $\dim V_{\mu} < \infty$ for all $\mu \in P$,
			
			\vskip 2mm
			
			\item[(ii)] there exist finitely many $\lambda_{1}, \ldots, \lambda_{s} \in P$
			such that $\text{wt}(V) \subset \bigcup_{j=1}^{s} (\lambda_{j} -  R_{+})$,
			
			\vskip 2mm
			
			\item[(iii)]
			${\mathtt f}_{il} (V_{\mu}) \subset V_{\mu - l \alpha_{i}} \ \ \text{for all} \
			\mu \in P, \  (i,l) \in I^{\infty}$.
			
		\end{enumerate}
	}
	
\end{definition}

\vskip 3mm

For a non-zero vector $v \in V$, define $d_{il}(v)$ to be the non-negative integer $n$
such that $v \in {\mathtt f}_{il}^n V \setminus {\mathtt f}_{il}^{n+1} V$.

\vskip 3mm

\begin{definition} \label{def:lower perfect basis}
	
	{\rm
		
		Let $(V, \{ {\mathtt f}_{il} \}_{(i,l) \in I^{\infty}})$ be a weak lower perfect space.
		
		\vskip 2mm
		
		\begin{enumerate}
			
			\item[(a)] A basis $\B$ of $V$ is called a {\it lower perfect basis} if
			
			\begin{enumerate}
				
				\item[(i)] $\B = \bigsqcup_{\mu \in P} \B_{\mu}$, where
				$\B_{\mu} = \B \cap V_{\mu}$,
				
				\vskip 2mm
				
				\item[(ii)] For every $(i,l) \in I^{\infty}$, there exists a map
				${\mathbf f}_{il}: \B \rightarrow \B \cup \{0\}$ such that
				for each $b \in \B$, there exists $c = c_{b} \in {\mathbf k}^{\times}$
				satisfying
				\begin{equation*}
					{\mathtt f}_{il}(b) - c \, {\mathbf f}_{il}(b) \in {\mathtt f}_{il}^{d_{il}(b) + 2} V,
				\end{equation*}
				
				\vskip 2mm
				
				\item[(iii)] ${\mathbf f}_{il}(b) = {\mathbf f}_{il}(b')$ for all $(i,l)\in I^{\infty}$ implies $b = b'$.
				
			\end{enumerate}
			
			\vskip 3mm

			\item[(b)] $V$ is called a {\it lower perfect space} if it has a lower perfect basis.
			
		\end{enumerate}

	}
	
\end{definition}

\vskip 3mm

\begin{proposition} \label{prop:exist lps}\hfill
	\vskip 2mm
	{\rm
		
		\begin{enumerate}
			
			\item[(a)] The lower global basis $\B(\infty)$ of $U_{q}^{-}(\g)$ is a lower perfect basis.
			
			\vskip 2mm
			
			\item[(b)] The lower global basis $\B(\lambda)$ of $V(\lambda)$ is a lower perfect basis.
			
		\end{enumerate}
		
	}
	
\end{proposition}

\begin{proof} \
	For each $(i,l) \in I^{\infty}$, define
	\begin{equation*}
		\begin{aligned}
			& {\mathtt f}_{il}= {\widetilde f}_{il}: U_{q}^{-}(\g) \rightarrow U_{q}^{-}(\g), \\
			&{\mathbf f}_{il} \, G(b)  = G({\widetilde f}_{il} b) \ \ \text{for} \ b \in B(\infty).
		\end{aligned}
	\end{equation*}
	
	\vskip 2mm
	
	Then it is easy to verify that the conditions (i) and (ii) are satisfied with $c = 1$.
	
	\vskip 2mm
	
	If ${\mathbf f}_{il}(G(b)) = {\mathbf f}_{il}(G(b'))$ for $b, b' \in B(\infty)$,
	then $G({\widetilde f}_{il} b) = G({\widetilde f}_{il} b')$. Since $G$ is a linear isomorphism, we obtain
	${\widetilde f}_{il} b = {\widetilde f}_{il} b'$.
	Hence $b = b'$ and $\B(\infty)$ is a lower perfect basis of $U_{q}^{-}(\g)$.
	
	\vskip 3mm
	
	By the same argument, we can show that $\B(\lambda)$ is a lower
	perfect basis of $V(\lambda)$.
\end{proof}

\vskip 3mm

\begin{lemma} \label{lem:basic}

	{\rm
		Let $V$ be a $P$-graded vector space with a lower perfect basis $\B$.
		
		\vskip 2mm
		
		\begin{enumerate}
			
			\item[(a)] For any $b \in  \B$, $n \in \Z_{\ge 0}$ and $(i,l) \in I^{\infty}$,
			there exists $c \in {\mathbf k}^{\times}$ such that
			
			\begin{equation*}
				{\mathtt f}_{il}^n (b) - c\, {\mathbf f}_{il}^n (b) \in {\mathtt f}_{il}^{d_{il}(b) + n+ 1} \, V.
			\end{equation*}
			
			\vskip 2mm
			
			\item[(b)] For any $(i,l) \in I^{\infty}$ and $n \in \Z_{\ge 0}$, we have
			
			\begin{equation*}
				{\mathtt f}_{il}^n \, V = \bigoplus_{b \in {\mathbf f}_{il}^n \, \B} \, {\mathbf k} \, b
			\end{equation*}
			
			\vskip 2mm
			
			\item[(c)] For any $b \in \B$, we have
			
			\begin{equation*}
				d_{il}(b) = \max \{ n \ge 0 \mid b \in {\mathbf f}_{il}^n \, \B \}.
			\end{equation*}
			
			\vskip 2mm
			
			\item[(d)] For any $b \in \B$ with ${\mathbf f}_{il} b \neq 0$, we have
			
			\begin{equation*}
				d_{il}({\mathbf f}_{il}b) = d_{il}(b) + 1.
			\end{equation*}
			
			\vskip 2mm
			
			\item[(e)] For any $n \ge 0$, the image of the set
			$\{ b + {\mathtt f}_{il}^{n+1}\, V \mid
			b \in \B, d_{il}(b) = n \}$
			is a basis of ${\mathtt f}_{il}^n \, V \big/ {\mathtt f}_{il}^{n+1} \, V$.
			
		\end{enumerate}
		
	}
	
\end{lemma}

\vskip 2mm

\begin{proof} \
	To prove (a), we will use  induction on $n$.
	
	\vskip 3mm
	
	By the definition of $d_{il}$, we have
	\begin{equation*}
		d_{il}({\mathbf f}_{il}\, b) \ge d_{il}(b) + 1 \ \ \text{for any} \
		b \in \B \ \ \text{with} \ {\mathbf f}_{il} b \neq 0.
	\end{equation*}
	
	\vskip 2mm
	
	\noindent
	Hence if $b \in \B$ and ${\mathbf f}_{il}^n (b) \neq 0$ for $n \ge 0$, then  we obtain
	\begin{equation} \label{eq:ind1}
		d_{il}({\mathbf f}_{il}^n b) \ge d_{il} ({\mathbf f}_{il}^{n-1} b) + 1
		\ge \cdots \ge d_{il}({\mathbf f}_{il} b) + n-1 \ge d_{il}(b) + n.
	\end{equation}

	\vskip 2mm
	
	If  $n=0, 1$, then (a) is trivial.
	
	\vskip 2mm
	
	Assume that $n \ge 2$.
	By the induction hypothesis, we have
	\begin{equation*}
		{\mathtt f}_{il}^{n-1}(b) - c \, {\mathbf f}_{il}^{n-1}(b)
		\in {\mathtt f}_{il}^{d_{il}(b) + n} V \ \ \text{for some} \ c \in {\mathbf k}^{\times}.
	\end{equation*}

	\vskip 2mm
	
	\noindent
	It follows that
	\begin{equation} \label{eq:ind2}
		{\mathtt f}_{il}^n b - c \, {\mathtt f}_{il} \, {\mathbf f}_{il}^{n-1} b
		\in {\mathtt f}_{il}^{d_{il}(b) + n+1} \, V.
	\end{equation}
	
	\vskip 2mm
	
	If ${\mathbf f}_{il}^{n-1} b =0$, we are done.
	
	\vskip 2mm
	
	If ${\mathbf f}_{il}^{n-1} b \neq 0$,
	by the definition of lower perfect bases and \eqref{eq:ind1},
	there exists $c' \in {\mathbf k}^{\times}$ such that
	\begin{equation} \label{eq:ind3}
		{\mathtt f}_{il} ({\mathbf f}_{il}^{n-1} b) - c' \, {\mathbf f}_{il}^n (b)
		\in {\mathtt f}_{il}^{d_{il}({\mathbf f}_{il}^{n-1}b) + 2} \, V
		\subset {\mathtt f}_{il}^{d_{il}(b) + (n-1) + 2} \, V
		= {\mathtt f}_{il}^{d_{il}(b) + n + 1} \, V.
	\end{equation}
	
	\vskip 3mm
	
	Combining \eqref{eq:ind2} and \eqref{eq:ind3}, we obtain the
	desired result.
	
	\vskip 3mm
	
	(b) Since (a) holds for any $b  \in \B$, $n \ge 0$ and $(i,l) \in I^{\infty}$, we have
	\begin{equation*}
		\bigoplus_{b \in {\mathbf f}_{il}^n \B} {\mathbf k} \, b \subset {\mathtt f}_{il}^n \,V.
	\end{equation*}
	
	Moreover, since $\B$ is a basis of $V$, we have
	\begin{equation*}
		{\mathtt f}_{il}^n V \,
		\subset \, \bigoplus_{b \in {\mathbf f}_{il}^{n} \B} \, {\mathbf k}b
		+ {\mathtt f}_{il}^{d_{il}(b) + n+1} V
		\, \subset \, \bigoplus_{b \in {\mathbf f}_{il}^{n} \B} \, {\mathbf k}b
		+ {\mathtt f}_{il}^{n+1} V.
	\end{equation*}
	
	\vskip 3mm
	
	Note that for any $\mu \in P$, we have $({\mathtt f}_{il}^n V)_{\mu} = 0$ for $n \gg 0$. Then, we assume  $({\mathtt f}_{il}^n V)_{\mu} = 0$ for all $n \ge k + 1$.
	It would yield
	
	\begin{equation*}
		({\mathtt f}_{il}^n \, V)_{\mu} =  \left(\bigoplus_{b \in ({\mathbf f}_{il}^n \B)_{\mu}}  \, {\mathbf k} \, b \right)
		\oplus \left(\bigoplus_{b' \in ({\mathbf f}_{il}^{n+1} \B)_{\mu}}  \, {\mathbf k} \, b' \right)
		\oplus \cdots  \oplus \left(\bigoplus_{b'' \in ({\mathbf f}_{il}^{k} \B)_{\mu}}  \, {\mathbf k} \, b'' \right).
	\end{equation*}
	
	\vskip 2mm
	
	Since
	$${\mathtt f}_{il}^n \, V \supset {\mathtt f}_{il}^{n+1} \, V
	\supset {\mathtt f}_{il}^{n+2}\, V \supset \cdots,$$
	we conclude
	$${\mathtt f}_{il}^n \, V = \bigoplus_{b \in {\mathbf f}_{il}^n \B} \, {\mathbf k} b,$$
	which is the statement (b).
	
	\vskip 3mm
	
	The statements (c), (d), (e) follow from (b).
\end{proof}

\vskip 3mm

Let $\B$ be a lower perfect basis of a $P$-graded space $V = \bigoplus_{\mu \in P} V_{\mu}$.

\vskip 2mm

For $b \in \B$, we define

\begin{equation*}
	{\mathbf e}_{il}\, b =
	\begin{cases} b' \ \ & \text{if} \ b = {\mathbf f}_{il}\, b'.\\
		0 \ \ & \text{if} \ b \notin {\mathbf f}_{il} \, \B.
	\end{cases}
\end{equation*}

\vskip 3mm

We define the maps $\text{wt}: \B \rightarrow P$, \ $\varepsilon_{i}, \varphi_{i}:\B \rightarrow \Z \cup \{-\infty\}$
by

\begin{equation*}
	\begin{aligned}
		& \text{wt}(b) = \begin{cases} \mu \ \ & \text{if} \ b \in \B_{\mu}, \\
			0 \ \ & \text{if} \ b \notin \B_\mu,
		\end{cases} \\
		& \varepsilon_{i}(b) = \begin{cases}
		 d_{i}(b) \ \ & \text{if} \ i \in I^{\text{re}}, \\
		 0 \ \ & \text{if} \ i \in I^{\text{im}},
		\end{cases} \\
		& \varphi_{i}(b) = \varepsilon_{i}(b) + \langle h_{i}, \text{wt}(b) \rangle.
	\end{aligned}
\end{equation*}

\vskip 2mm

Then it is straightforward to verify that $(\B, {\mathbf e}_{il}, {\mathbf f}_{il}, \text{wt},
\varepsilon_{i}, \varphi_{i})$ is an abstract crystal, called the {\it lower perfect graph} of $\B$.

\vskip 3mm

\begin{proposition} \label{prop:lower crystal} \hfill
	\vskip 2mm
	
	{\rm
		
		\begin{enumerate}
			
			\item[(a)] The lower perfect graph of $\B(\infty)$ is isomorphic to $B(\infty)$.
			
			\vskip 2mm
			
			\item[(b)]  For $\lambda \in P^{+}$, the lower perfect graph of $\B(\lambda)$ is isomorphic to $B(\lambda)$.
			
		\end{enumerate}
		
	}
\end{proposition}

\begin{proof} \ Let $b, b' \in B(\infty)$ such that $\widetilde{f}_{il} b = b'$.
	Since $G$ is a $\Q$-linear isomorphism, we have $G(\widetilde{f}_{il} b) = G(b')$.
	The definition of ${\mathbf f}_{il}$ yields
	${\mathbf f}_{il} \, G(b) = G(b').$
	
	\vskip 2mm
	
	Conversely, if ${\mathbf f}_{il} \, G(b) = G(b')$ for $b, b' \in B(\infty)$, it is clear that
	$\widetilde{f}_{il}b = b'$.
	
	\vskip 3mm
	
	By the same argument, the statement (b) follows.
\end{proof}

\vskip 3mm

Let $\B$ be a lower perfect basis of  $V=\bigoplus_{\mu \in P} V_{\mu}$.
For $b \in \B$ and $(i,l)\in I^{\infty}$, define

\begin{equation} \label{eq:etop}
	{\mathbf e}_{il}^{\text{top}}(b) := {\mathbf e}_{il}^{d_{il}(b)}(b) \in \B.
\end{equation}

\vskip 2mm

\noindent
More generally, for a sequence ${\mathbf i} =((i_1, l_1), \ldots, (i_r, l_r))$, we define

\begin{equation} \label{eq:etopseq}
	{\mathbf e}_{\mathbf{i}}^{\text{top}} (b): = {\mathbf e}_{i_r l_r}^{\text{top}} \
	\cdots \ {\mathbf e}_{i_1 l_1}^{\text{top}}(b).
\end{equation}

\vskip 2mm

In addition, for a sequence ${\mathbf a} = (a_{1}, \ldots, a_{r}) \in \Z_{\ge 0}^r$,
we define

\begin{equation} \label{eq:efseq}
	\begin{aligned}
		& {\mathbf e}_{\mathbf{i}, \mathbf{a}}(b) = {\mathbf e}_{i_r l_r}^{a_r} \cdots
		{\mathbf e}_{i_1 l_1}^{a_1} (b) \in \B \cup \{0\}, \\
		& {\mathbf f}_{\mathbf{i}, \mathbf{a}}(b) = {\mathbf f}_{i_1 l_1}^{a_1} \cdots
		{\mathbf f}_{i_r l_r}^{a_r} (b) \in \B \cup \{0\}.
	\end{aligned}
\end{equation}

\vskip 3mm

We will also use the notation

\begin{equation} \label{eq:fseq}
	{\mathtt f}_{\mathbf{i}, \mathbf{a}} = {\mathtt f}_{i_1 l_1}^{a_1} \cdots {\mathtt f}_{i_r l_r}^{a_r}.
\end{equation}

\vskip 3mm

\begin{definition} \label{def:good seq} \hfill
	
	\vskip 2mm
	
	{\rm
		
		\begin{enumerate}
			
			\item[(a)] A sequence of non-negative integers ${\mathbf a} = (a_{k})_{k \ge 1}$
			is {\it good} if $a_{k} = 0$ for $k \gg 0$.
			
			\vskip 2mm
			
			\item[(b)]  A sequence of indices ${\mathbf i} = ((i_k, l_k))_{k \ge 1}$ is {\it good}
			if every $(i,l) \in I^{\infty}$ appears infinitely many times in ${\mathbf i}$.

		\end{enumerate}
		
	}
\end{definition}

\vskip 2mm

Let $\mathbf{i} = ((i_k, l_k))_{k \ge 1}$ be a good sequence of indices and let
${\mathbf a} = (a_{k})_{k \ge 1}$ be a good sequence of non-negative integers
with $a_{k} = 0$ for all $k \ge r+1$.
Define
\begin{equation*}
	\begin{aligned}
		& V^{> \mathbf{i}, \mathbf{a}} = \sum_{k=1}^r {\mathtt f}_{i_1 l_1}^{a_1} \cdots
		{\mathtt f}_{i_{k-1}l_{k-1}}^{a_{k-1}}
		{\mathtt f}_{i_k l_k}^{a_k+1} \, V, \\
		& V^{\ge \mathbf{i}, \mathbf{a}} = V^{>\mathbf{i}, \mathbf{a}}
		+ {\mathtt f}_{i_1 l_1}^{a_1} \cdots {\mathtt f}_{i_r l_r}^{a_r} \, V.
	\end{aligned}
\end{equation*}

\vskip 2mm

The set of good sequences of non-negative integers  is endowed with
a total ordering given as follows.
Let ${\mathbf a} = (a_{k})_{k \ge 1}$ and ${\mathbf a}' = (a_{k}')_{k \ge 1}$ be good sequences
of non-negative integers.
We define ${\mathbf a}' < {\mathbf a}$ if and only if there exists some $m\le r$ such that
$a_{s}' = a_{s}$ for all $s<m$ and $a_{m}' < a_{m}$.

\vskip 2mm

Under this ordering, if ${\mathbf a}' < {\mathbf a}$, then we have

\begin{equation*}
	V^{\ge \mathbf{i}, \mathbf{a}} \subset V^{\ge \mathbf{i}, \mathbf{a}'}
	\ \ \text{and} \ \
	V^{> \mathbf{i}, \mathbf{a}} \subset V^{\ge \mathbf{i}, \mathbf{a}'}.
\end{equation*}

\vskip 2mm

Let $v \in V \setminus \{0\}$ and let ${\mathbf i} = ((i_k, l_k))_{k \ge 1}$
be a good sequence of indices.
We denote by ${\mathbf L}(\mathbf{i}, v)$ the largest sequence of
non-negative integers $\mathbf{a}$
such that $v \in V^{\ge \mathbf{i}, \mathbf{a}}$.
In the next proposition, we will show that ${\mathbf L}(\mathbf{i}, v)$ exists
and it is a good sequence.
%Thus we have $v \in V^{\ge \mathbf{i}, \mathbf{a}}$ if and only if
%$\mathbf{a} \le {\mathbf L}(\mathbf{i}, v)$.

\vskip 3mm

We define the {\it highest core} of $\B$ to be
\begin{equation*}
	\B_{H} = \{ b \in \B \mid d_{il}(b) = 0 \ \ \text{for all} \ (i,l) \in I^{\infty} \}.
\end{equation*}

\vskip 3mm

\begin{proposition} \label{prop:hwcore} 
	{\rm
		
		Let ${\mathbf i} = ((i_k, l_k))_{k \ge 1}$ be a  good sequence of indices.
		
		\vskip 2mm
		
		\begin{enumerate}
			
			\item[(a)] For any $b \in \B$ and a sequence ${\mathbf a} = (a_1, \ldots, a_r) \in \Z_{\ge 0}^{r}$,
			there exists $c \in {\mathbf k}^{\times}$ such that
			\begin{equation*}
				{\mathtt f}_{i_1 l_1}^{a_1} \cdots {\mathtt f}_{i_r l_r}^{a_r} b
				- c \, {\mathbf f}_{i_1 l_1}^{a_1} \cdots {\mathbf f}_{i_r l_r}^{a_r}  b \in \,
				\sum_{k=1}^{r} {\mathtt f}_{i_1 l_1}^{a_1} \cdots {\mathtt f}_{i_r l_k}^{a_k+1} \, V
				= V^{> \mathbf{i}, \mathbf{a}}.
			\end{equation*}

			\vskip 2mm
			
			\item[(b)] For each $b \in \B$, define the sequences ${\mathbf b}= (b_k)_{k \ge 1}$
			and ${\mathbf d} (\mathbf{i}, b)= (d_{k})_{k \ge 1}$ by			
			\begin{equation}\label{eq:bk dk}
				b_{0} = b, \ \ b_{k} = {\mathbf e}_{i_k l_k}^{\text{top}}(b_{k-1}),
				\quad
				d_{k} = d_{i_k l_k} (b_{k-1}) \ \ \text{for} \ k \ge 1.
			\end{equation}
			
			\vskip 2mm
			
			Then obtain
			
			\vskip 2mm
			
			\begin{enumerate}
				
				\item[(i)] ${\mathbf d}(\mathbf{i}, b)$ is a good sequence.
				
				\vskip 2mm
				
				\item[(ii)] $b_{k} \in \B_{H}$ for $k \gg 0$.
				
				\vskip 2mm
				
				\item[(iii)] ${\mathbf d}(\mathbf{i}, b) = {\mathbf L}(\mathbf{i}, b)$.
				
			\end{enumerate}
			
			\vskip 2mm
			
			\item[(c)] For any good sequence ${\mathbf a} = (a_k)_{k \ge 1}$, we have
			
			\begin{equation*}
				V^{\ge \mathbf{i}, \mathbf{a}} = \sum_{b \in \B \atop  {\mathbf d}(\mathbf{i}, b) \ge \mathbf{a} }
				\mathbf{k} \, b  \quad \ \text{and} \quad \
				V^{> \mathbf{i}, \mathbf{a}} = \sum_{b \in \B \atop  {\mathbf d}(\mathbf{i}, b) > \mathbf{a} }
				\mathbf{k} \, b.
			\end{equation*}
			
			\vskip 2mm
			
			\item[(d)] For any good sequence $\mathbf{a} = (a_{k})_{k \ge 1}$ with $a_{k} = 0$ for all $k \ge r+1$,
			define
			
			\begin{equation*}
				\B_{\mathbf{i}, \mathbf{a}} = \{ b\in \B \mid {\mathbf d}(\mathbf{i}, b) = \mathbf{a} \}.
			\end{equation*}
			
			\vskip 2mm
			
			Then we have an injective map			
			\begin{equation*}
				{\mathbf e}_{\mathbf{i}, \mathbf{a}} : \B_{\mathbf{i}, \mathbf{a}} \longrightarrow \B_{H}
				\quad \text{given by} \ \  b \mapsto
				{\mathbf e}_{\mathbf{i}, \mathbf{a}}(b) = {\mathbf e}_{i_r l_r}^{a_r} \cdots {\mathbf e}_{i_1 l_1}^{a_1} (b).
			\end{equation*}

		\end{enumerate}
		
	}
	
\end{proposition}

\vskip 2mm

\begin{proof} \ (a) We will use induction on $r$.
	If $r=1$, our assertion follows from Lemma \ref{lem:basic} (a).
	
	\vskip 2mm
	
	Assume that $r \ge 2$.
	The induction hypothesis gives
	\begin{equation*}
		{\mathtt f}_{i_1 l_1}^{a_1} \cdots {\mathtt f}_{i_{r-1} l_{r-1}}^{a_{r-1}} ({\mathbf f}_{i_rl_r}^{a_r} b)
		- c \, {\mathbf f}_{i_1 l_1}^{a_1} \cdots {\mathbf f}_{i_{r-1} l_{r-1}}^{a_{r-1}} ({\mathbf f}_{i_r l_r}^{a_r} b)
		\, \in \, \sum_{k=1}^{r-1} {\mathtt f}_{i_1 l_1}^{a_1} \cdots {\mathtt f}_{i_k l_k}^{a_k + 1} \, V
	\end{equation*}
	for some $c \in {\mathbf k}^{\times}$.
	
	\vskip 3mm
	
	Also, by Lemma \ref{lem:basic} (a), there exists $c' \in {\mathbf k}^{\times}$ such that
	
	\begin{equation*}
		{\mathtt f}_{i_r l_r} ^{a_r}- c' \, {\mathbf f}_{i_r l_r}^{a_r} b \in {\mathtt f}_{i_r l_r}^{a_r + 1} \, V.
	\end{equation*}
	
	\vskip 3mm
	
	Thus there exist $w$, $w'$ such that
	\begin{equation*}
		w \in \sum_{k=1}^{r-1} {\mathtt f}_{i_1 l_1}^{a_1} \cdots {\mathtt f}_{i_{k-1} l_{k-1}}^{a_{k-1}}
		\mathtt f_{i_k l_k}^{a_{k}+1} \, V \quad \text{and} \quad
		w' \in {\mathtt f}_{i_r l_r}^{a_{r}+1} V
	\end{equation*}
	satisfying
	\begin{equation*}
		\begin{aligned}
			& {\mathtt f}_{i_1 l_1}^{a_1} \cdots {\mathtt f}_{i_{r-1} l_{r-1}}^{a_{r-1}} ({\mathbf f}_{i_r l_r}^{a_r} b)
			= c \, {\mathbf f}_{i_1 l_1}^{a_1} \cdots {\mathbf f}_{i_{r-1} l_{r-1}}^{a_{r-1}} ({\mathbf f}_{i_r l_r}^{a_r} b) + w, \\
			& {\mathtt f}_{i_r l_r}^{a_r} b = c' \, {\mathbf f}_{i_r l_r}^{a_r} b + w'.
		\end{aligned}
	\end{equation*}
	
	\vskip 2mm
	
	Therefore we obtain
	
	\begin{equation*}
		\begin{aligned}
			{\mathtt f}_{i_1 l_1}^{a_1} &  \cdots {\mathtt f}_{i_r l_r}^{a_r} b
			= {\mathtt f}_{i_1 l_1}^{a_1} \cdots {\mathtt f}_{i_{r-1} l_{r-1}}^{a_{r-1}}
			(c' \, {\mathbf f}_{i_r l_r}^{a_r} b + w')\\
			& = c' \, ({\mathtt f}_{i_1 l_1}^{a_1} \cdots {\mathtt f}_{i_{r-1} l_{r-1}}^{a_{r-1}}
			{\mathbf f}_{i_r l_r}^{a_r} b)
			+ {\mathtt f}_{i_1 l_1}^{a_1} \cdots {\mathtt f}_{i_{r-1} l_{r-1}}^{a_{r-1}} w' \\
			& = c' (c \, {\mathbf f}_{i_1 l_1}^{a_1} \cdots {\mathbf f}_{i_r l_r}^{a_r} b + w) + {\mathtt f}_{i_1 l_1}^{a_1} \cdots {\mathtt f}_{i_{r-1} l_{r-1}}^{a_{r-1}} w' \\
			& = c' c \, {\mathbf f}_{i_1 l_1}^{a_1} \cdots {\mathbf f}_{i_r l_r}^{a_r} b + (c'w + {\mathtt f}_{i_1 l_1}^{a_1} \cdots {\mathtt f}_{i_{r-1} l_{r-1}}^{a_{r-1}} w').
		\end{aligned}
	\end{equation*}
	
	\vskip 2mm
	
	Note that
	\begin{equation*}
		c'w +{\mathtt f}_{i_1 l_1}^{a_1} \cdots {\mathtt f}_{i_{r-1} l_{r-1}}^{a_{r-1}} w' \in \sum_{k=1}^{r} {\mathtt f}_{i_1 l_1}^{a_1} \cdots f_{i_k l_k}^{a_k + 1} V.
	\end{equation*}
	Hence our claim follows.
	
	\vskip 3mm
	
	(b) By the definitions in \eqref{eq:etop} and \eqref{eq:bk dk}, we have
	
	\begin{equation*}
		b_k={\mathbf e}_{i_kl_k}^{d_{i_kl_k}(b_{k-1})}{\mathbf e}_{i_{k-1}l_{k-1}}^{d_{i_{k-1}l_{k-1}}(b_{k-2})}\cdots {\mathbf e}_{i_1l_1}^{d_{i_1l_1}(b_{0})}(b_0).
	\end{equation*}
	
	Thus we obtain	
	\begin{equation*}
		\text{wt}(b_k) = \text{wt}(b) + \sum_{j = 1}^k d_{j}(l_j \alpha_{i_j}).
	\end{equation*}
	
	\vskip 2mm
	
	Since $V$ is a lower perfect space, $\text{wt}(b_k)$ is stable for $k \gg 0$.
	Hence we have $d_{k} = 0$ for $k \gg 0$.
	It follows that ${\mathbf d}(\mathbf{i}, b)$ is a good sequence.
	
	\vskip 2mm
	
	By (i), we have $d_{i_k l_k}(b_{k-1}) = 0$ for $k \gg 0$.
	Since $\mathbf{i}$ is a good sequence,
	every $(i,l)$ appears infinitely times in $\mathbf{i}$.
	Thus we conclude $d_{il}(b_{k}) = 0$ for all $(i,l)\in I^{\infty}$ and  $k \gg 0$,
	which gives the statement (ii).
	
	\vskip 3mm
	
	We prove the statements (b)-(iii) and (c) simultaneously.
	Let $\mathbf{a} = (a_1, \ldots, a_r)$ be a sequence of non-negative integers
	and let $\mathbf{d}_{r} = {\mathbf d}_{r}(\mathbf{i}, b)
	= (d_1, \ldots d_r)$.
	By (a), we have
	\begin{equation*}
		b \in \sum_{k=1}^{r-1} {\mathtt f}_{i_1 l_1}^{a_1} \cdots {\mathtt f}_{i_k l_k}^{a_k + 1} V
		+ {\mathtt f}_{i_1 l_1}^{a_1} \cdots {\mathtt f}_{i_r l_r}^{a_r} V.
	\end{equation*}

	Hence we obtain
	\begin{equation} \label{eq:LHS}
		\sum_{b \in \B \atop \mathbf{a} \le \mathbf{d}_{r}} \mathbf{k} \, b \subset
		\sum_{k=1}^{r-1}  {\mathtt f}_{i_1 l_1}^{a_1} \cdots {\mathtt f}_{i_k l_k}^{a_k + 1} V
		+ {\mathtt f}_{i_1 l_1}^{a_1} \cdots {\mathtt f}_{i_r l_r}^{a_r} V.
	\end{equation}
	
	\vskip 2mm
	
	We claim that
	\begin{equation} \label{eq:claim}
		\sum_{b \in \B \atop \mathbf{a} \le \mathbf{d}_{r}} \mathbf{k} \, b =
		\sum_{k=1}^{r-1}  {\mathtt f}_{i_1 l_1}^{a_1} \cdots {\mathtt f}_{i_k l_k}^{a_k + 1} V
		+ {\mathtt f}_{i_1 l_1}^{a_1} \cdots {\mathtt f}_{i_r l_r}^{a_r} V.
	\end{equation}
	
	\vskip 2mm
	
	To prove our claim, it suffices to show that
	
	\begin{equation} \label{eq:RHS}
		{\mathtt f}_{i_1 l_1}^{a_1}  \cdots {\mathtt f}_{i_r l_r}^{a_r} V \subset
		\sum_{b \in \B \atop \mathbf{a} \le \mathbf{d}_{r}} \mathbf{k} \, b.
	\end{equation}
	
	\vskip 2mm
	
	We will use induction $r$.
	When $r=1$, \eqref{eq:RHS} is trivial.
	Assume that $r \ge 2$ and set
	
	\begin{equation*}
		{\mathbf i}' = ((i_k, l_k))_{k \ge 2}, \ \ {\mathbf a}' = (a_2.\ldots, a_r) \ \ \text{and} \ \
		{\mathbf d}_{r}' = (d_2, \ldots, d_r).
	\end{equation*}
	
	\vskip 2mm
	
	The induction hypothesis gives
	
	\begin{equation*}
		{\mathtt f}_{i_2 l_2}^{a_2} \cdots {\mathtt f}_{i_r l_r}^{a_r} V
		\subset \sum_{b \in \B \atop {\mathbf a}' \le {\mathbf d}_{r}'} {\mathbf k} \, b.
	\end{equation*}
	
	\vskip 2mm
	
	Hence we have only to show that
	\begin{equation} \label{eq:LHS-RHS}
		{\mathtt f}_{i_1 l_1}^{a_1}  b \in \sum_{b' \in \B \atop {\mathbf a} \le {\mathbf d}_{r}(\mathbf{i}, b')}
		\mathbf{k} \, b'.
	\end{equation}
	
	\vskip 2mm
	
	Note that
	
	\vskip 2mm
	
	\begin{enumerate}
		
		\item[(i)] $\mathbf{d}_{r}(\mathbf{i}, {\mathbf f}_{i_1 l_1}^{a_1} b) \ge
		{\mathbf d}_{r}(\mathbf{i}, b) \ge \mathbf{a} = (a_1, \ldots, a_r)$.
		
		\vskip 2mm
		
		\item[(ii)] ${\mathtt f}_{i_1 l_1}^{a_1} b \in {\mathbf k}  \, {\mathbf f}_{i_1 l_1}^{a_1} b
		+ \sum_{d_{il}(b') > a_{1}} {\mathbf k} \, b' $.
		
	\end{enumerate}
	
	\vskip 2mm
	
	Thus  \eqref{eq:LHS-RHS} holds and \eqref{eq:RHS} follows.
	Therefore the statements (b)-(iii) and (c) are proved.
	
	\vskip 3mm
	
	(d) Since $\mathbf{d}(\mathbf{i}, b) = \mathbf{L}(\mathbf{i}, b)$,
	it is easy to see that ${\mathbf e}_{\mathbf{i}}^{\text{top}}(b)
	= {\mathbf e}_{\mathbf{i}, \mathbf{a}}(b)$.
	Hence ${\mathbf f}_{\mathbf{i}, \mathbf{a}} \circ {\mathbf e}_{\mathbf{i}, \mathbf{a}}
	= \text{id}|_{\B_{H}}$, which implies ${\mathbf e}_{\mathbf{i}, \mathbf{a}}$ is injective.
\end{proof}

\vskip 3mm

Let $\mathbf{i} = ((i_k, l_k))_{k \ge 1}$ and $\mathbf{a} = (a_{k})_{k \ge 1}$ be good sequences
and let

\begin{equation*}
	p_{\mathbf{i}, \mathbf{a}} : V^{\ge \mathbf{i}, \mathbf{a}} \longrightarrow
	V^{\ge \mathbf{i}, \mathbf{a}} \big/ V^{> \mathbf{i}, \mathbf{a}}
\end{equation*}
be the canonical projection.

\vskip 3mm

Define $$V_{H} = V \big / (\sum_{(i,l) \in I^{\infty}}{\mathtt f}_{il} \,V)$$
and let $$p_{H}: V \longrightarrow V_{H}$$ be the canonical projection.

\vskip 3mm

For a subset $S$ of a vector space $V$, we denote

\begin{equation*}
	\mathbf{k}^{\times} \, S = \{a u \mid a \in \mathbf{k}^{\times}, \ u \in S\}.
\end{equation*}

\vskip 3mm

With the above notations, we obtain the following corollary.

\vskip 3mm

\begin{corollary} \label{cor:lemma} \hfill
	
	{\rm
		
		\vskip 2mm
		
		\begin{enumerate}
			
			\item[(a)] The image  $p_{\mathbf{i}, \mathbf{a}}(\B_{\mathbf{i}, \mathbf{a}})$
			is a basis of $V^{\ge \mathbf{i}, \mathbf{a}} \big / V^{> \mathbf{i}, \mathbf{a}}$.
			
			\vskip 2mm
			
			\item[(b)] We have  $\mathbf{k}^{\times} \, p_{\mathbf{i}, \mathbf{a}}(\B_{\mathbf{i}, \mathbf{a}})
			= \mathbf{k}^{\times} (p_{\mathbf{i}, \mathbf{a}}({\mathtt f}_{\mathbf{i}, \mathbf{a}} \B_{H}) \setminus\{0\})$.
			
			\vskip 2mm
			
			\item[(c)] The map $p_{H}:\B_{H} \rightarrow V_{H}$ is injective.
			
			\vskip 2mm
			
			\item[(d)] $p_{H}(\B_{H})$ is a basis of $V_{H}$.

		\end{enumerate}
		
	}
	
\end{corollary}

\vskip 2mm

\section{Uniqueness of lower perfect graphs} \label{sec:unique}

\vskip 2mm

Let $V$ be a lower perfect space.
In this section, we will show that all the lower perfect graphs arising from lower perfect bases of $V$
are isomorphic.

\vskip 3mm

\begin{theorem} \label{thm:unique crystal}

	{\rm
		
		Let $V$ be a lower perfect space and let $\B$, $\B'$ be lower perfect bases of $V$.
		Assume that $p_{H}(\B_{H}) = p_{H}(\B'_{H})$.
		Then the following statements hold.
		
		\vskip 2mm
		
		\begin{enumerate}
			
			\item[(a)] There exists a crystal isomorphism
			\begin{equation*}
				\psi: \B  \overset{\sim} \longrightarrow \B' \ \ \text{such that} \ \
				p_{H}(b) = p_{H}(\psi(b)) \ \ \text{for all} \ b \in \B_{H}.
			\end{equation*}
			
			\vskip 2mm
			
			\item[(b)] For any $b \in \B$ and any good sequence $\mathbf{i} $,  we have
			\begin{equation*}
				\mathbf{d}(\mathbf{i}, b) = \mathbf{d}(\mathbf{i}, \psi(b)), \ \
				p_{H}(\mathbf{e}_{\mathbf{i}}^{\text{top}}(b))
				= p_{H}(\mathbf{e}_{\mathbf{i}}^{\text{top}} \psi(b)) \in p_{H}(\B_{H}).
			\end{equation*}
			
		\end{enumerate}
		
	}
\end{theorem}

\vskip 2mm

\begin{proof} \ Let $\mathbf{i} = ((i_k, l_k))_{k \ge 1}$ be a good sequence of indices
	and let $\mathbf{a} = (a_k)_{k \ge 1}$ be a good sequence of non-negative integers.
	
	\vskip 3mm
	
	Let
	\begin{equation*}
		p_{\mathbf{i}, \mathbf{a}}: V^{\ge \mathbf{i}, \mathbf{a}} \longrightarrow
		V^{\ge \mathbf{i}, \mathbf{a}} \big/ V^{> \mathbf{i}, \mathbf{a}}
	\end{equation*}
	be the canonical projection and set
	\begin{equation*}
		\B_{\mathbf i, \mathbf a} = \{ b \in \B \mid \mathbf{d}(\mathbf{i}, b) = \mathbf{a} \},
		\ \ \B'_{\mathbf i, \mathbf a} = \{ b' \in \B' \mid \mathbf{d}(\mathbf{i}, b') = \mathbf{a} \}.
	\end{equation*}
	
	\vskip 3mm
	
	Then Corollary \ref{cor:lemma} yields:
	
	\vskip 2mm
	
	\begin{enumerate}
		
		\item[(a)] $\mathbf{k}^{\times} p_{\mathbf{i}, \mathbf{a}} (\B_{\mathbf{i}, \mathbf{a}})
		= \mathbf{k}^{\times} p_{\mathbf{i}, \mathbf{a}} (\B'_{\mathbf{i}, \mathbf{a}})$.
		
		\vskip 2mm
		
		\item[(b)] Both $p_{\mathbf{i}, \mathbf{a}}(\B_{\mathbf{i}, \mathbf{a}})$
		and $p_{\mathbf{i}, \mathbf{a}}(\B'_{\mathbf{i}, \mathbf{a}})$ are bases of
		$V^{\ge \mathbf{i}, \mathbf{a}} \big/ V^{> \mathbf{i}, \mathbf{a}}$.
		
	\end{enumerate}
	
	\vskip 3mm
	
	Hence for any $b \in \B$, there exists $b' \in \B'$
	and $c \in \mathbf{k}^{\times}$ such that
	\begin{equation} \label{eq:B to B'}
		\mathbf{d}(\mathbf{i}, b) = \mathbf{d}(\mathbf{i}, b')
		\ \ \text{and}
		\ \  b - c \, b' \in V^{> \mathbf{i}, \mathbf{a}} = V^{> \mathbf{i}, \mathbf{d}(\mathbf{i}, b)}.
	\end{equation}
	
	\vskip 2mm
	
	We will prove \eqref{eq:B to B'} still holds for any choice of good sequence $\mathbf{i}'$
	with the same $b'$ so that we can define $\psi(b) = b'$.
	
	\vskip 3mm
	
	Set $v = b - c b' \in V^{> \mathbf{i}, \mathbf{d}(\mathbf{i}, b)}$. Then the vector $v$ is a linear combination
	of elements in $\B \setminus \{b\}$.
	
	\vskip 3mm
	
	Let $\mathbf{c} = \mathbf{d}(\mathbf{i}', b)$ and
	$\mathbf{c}' = \mathbf{d}(\mathbf{i}', b')$.
	Then $b \in V^{\ge \mathbf{i}', \mathbf{c}}$ and $b' \in V^{\ge \mathbf{i}', \mathbf{c}'}$.
	If $\mathbf{c} < \mathbf{c}'$, then $b' \in V^{> \mathbf{i}', \mathbf{c}}$ and hence
	$v - b \in V^{> \mathbf{i}', \mathbf{c}}$.
	Thus $v - b$ is a linear combination of elements in $\B \setminus \{b\}$.
	But since $v$ is also a linear combination of elements in $\B \setminus \{b\}$,
	we get a contradiction.
	
	\vskip 3mm
	
	On the other hand, if $\mathbf{c} > \mathbf{c}'$, by a similar argument, we also get a contradiction.
	Therefore we must have $\mathbf{c} = \mathbf{c}'$.
	
	\vskip 3mm
	
	It follows that both $p_{\mathbf{i}', \mathbf{c}}(b)$ and $p_{\mathbf{i}', \mathbf{c}}(b')$
	belongs to  $\mathbf{k}^{\times}  p_{\mathbf{i}', \mathbf{c}} (\B_{\mathbf{i}, \mathbf{c}})$.
	If $\mathbf{k}^{\times} p_{\mathbf{i}', \mathbf{c}}(b) \neq
	\mathbf{k}^{\times} p_{\mathbf{i}', \mathbf{c}}(b')$,
	then $v-b$ is a linear combination of the elements in $\B \setminus \{b\}$, which is a
	contradiction.
	Hence  $\mathbf{k}^{\times} p_{\mathbf{i}', \mathbf{c}}(b) =
	\mathbf{k}^{\times} p_{\mathbf{i}', \mathbf{c}}(b')$ and
	\eqref{eq:B to B'} also holds for $\mathbf{i}'$.
	
	\vskip 3mm
	
	By defining $\psi(b) = b'$, we can immediately deduce our assertions.
\end{proof}

\vskip 3mm

By Proposition \ref{prop:exist lps}, Proposition \ref{prop:lower crystal} and Theorem \ref{thm:unique crystal}, we have the following corollary.

\begin{corollary} \label{cor:lower crystal} \hfill
	
	\vskip 2mm
	
	{\rm
		
		(a) Every lower perfect graph of $U_{q}^{-}(\g)$ is isomorphic to $B(\infty)$.
		
		\vskip 2mm
		
		(b) For $\lambda \in P^{+}$, every  lower perfect graph of $V(\lambda)$ is isomorphic to $B(\lambda)$.

	}
\end{corollary}

\vskip 3mm

\begin{example}
	{\rm
		Let $\B_{\Q}$ be the primitive canonical basis \cite{FHKK} of $U_{q}^{-}(\g)$
		and let $\beta \in (\B_{\Q})_{-\alpha}$ with $\alpha \in R_{+}$.
		Then there exist $(i_1, l_1) \in I^{\infty}$, $\mathbf{c}_{1} \vdash l_1$
		and $\beta_{1} \in (\B_{\Q})_{-\alpha + l_1 \alpha_{i_1}; i_1, 0}$ such that
		\begin{equation*}
			{\mathtt b}_{i_1, \mathbf{c}_1} \beta_1 - \beta \in
			\bigoplus_{\beta' \in (\B_{\Q})_{-\alpha; i_1, \ge l_1 + 1}} \A_{\Q} \, \beta'.
		\end{equation*}
		
		\vskip 2mm
		
		If $\beta_1 \neq \mathbf{1}$, then there exist
		$(i_2, l_2) \in I^{\infty}$, $\mathbf{c}_{2} \vdash l_2$
		and $\beta_{2} \in (\B_{\Q})_{-\alpha + l_1 \alpha_{i_1} +  l_2 \alpha_{i_2}; i_2, 0}$ such that
		\begin{equation*}
			{\mathtt b}_{i_2, \mathbf{c}_2} \beta_2 - \beta_1 \in
			\bigoplus_{\beta'' \in (\B_{\Q})_{-\alpha + l_1 \alpha_{i_1} ; i_2, \ge l_2 + 1}} \A_{\Q} \, \beta''.
		\end{equation*}
		
		\vskip 2mm
		
		Continuing this process, we obtain a sequence
		\begin{equation*}
			(i_1,\mathbf c_1),\ (i_2,\mathbf c_2),\ \cdots,\ (i_r,\mathbf c_r),
		\end{equation*}
		such that
		\begin{equation*}
			\beta\stackrel{\mathtt b_{i_1,\mathbf c_1}}{\longleftarrow}\beta_1\stackrel{\mathtt b_{i_2,\mathbf c_2}}{\longleftarrow}\beta_2\longleftarrow\cdots
			%\stackrel{\mathtt b_{i_{r-1},\mathbf c_{r-1}}}{\longleftarrow}
			\longleftarrow
			\beta_{r-1}\stackrel{\mathtt b_{i_r,\mathbf c_r}}{\longleftarrow}\mathbf 1.
		\end{equation*}
		
		\vskip 2mm
		
		Define
		\begin{equation*}
			b_{\beta}:= \widetilde{f}_{i_1, \mathbf{c}_1} \cdots \widetilde{f}_{i_r, \mathbf{c}_r} \mathbf{1} \in B(\infty).
		\end{equation*}
		Then we obtain a map
		\begin{equation*}
			\phi:\B_{\Q} \longrightarrow \B(\infty)  \ \ \text{given by} \ \
			\beta \longmapsto G(b_{\beta}).
		\end{equation*}
		
		\vskip 2mm
		
		%Conversely, given an element
		%$b = \widetilde{f}_{i_1, \mathbf{c}_1} \cdots \widetilde{f}_{i_r, \mathbf{c}_r} \mathbf{1} \in B(\infty)$,
		%define $\psi: \B(\infty) \longrightarrow \B_{\Q}$ by
		%\begin{equation*}
		%G(b) \longmapsto {\mathtt b}_{i_1, \mathbf{c}_1} \cdots {\mathtt b}_{i_r, \mathbf{c}_r} \mathbf{1}.
		%\end{equation*}
		
		%It is clear that $\phi$ and $\psi$ are inverses to each other.
		By reversing the above process, we obtain the inverse of $\phi$, which would yield a precise 1-1 correspondence between $\mathbf B_{\mathbf Q}$ and $\mathbf B(\infty)$.
		Since the primitive canonical 
		basis $\B_{\Q}$  coincide with the lower global basis  $\B(\infty)$(cf. \cite{FHKK}), we may define the maps on $\B_{\Q}$ by
		\begin{equation*}
			\begin{aligned}
				& \mathbf{f}_{il} \beta = \mathbf{f}_{il}G(b_{\beta}), \ \ \mathbf{e}_{il} \beta = \mathbf{e}_{il} G(b_{\beta}), \\
				& \text{wt} (\beta) = \text{wt}(b_{\beta}), \ \
				\varepsilon_{i}(\beta) = \varepsilon_{i}(b_{\beta}), \ \
				\varphi_{i}(\beta) = \varphi_{i}(b_{\beta}).
			\end{aligned}
		\end{equation*}
		
		\vskip 2mm
		
		It is straightforward to see that the lower perfect graph of $\B_{\Q}$ is isomorphic to $B(\infty)$.

	}
\end{example}

\vskip 2mm

\section{Upper perfect bases} \label{sec:upper}

\vskip 2mm

In this subsection, we will define the notion of upper perfect bases and investigate
their properties.

\vskip 3mm

\begin{definition} \label{def:upper}
	{\rm
		Let $V = \bigoplus_{\mu \in P} V_{\mu}$ be a $P$-graded vector space and let $\{\mathtt{e}_{il}\}_{(i,l)\in I^{\infty}}$ be a family of endomorphisms of $V$.
		We say that $(V, \{\mathtt{e}_{il}\}_{(i,l)\in I^{\infty}})$ is a {\it weak upper perfect space}
		if
		
		\begin{enumerate}
			
			\item[(i)] $\dim V_{\mu} < \infty$ for all $\mu \in P$,
			
			\vskip 2mm
			
			\item[(ii)] there exist finitely many  $\lambda_1, \ldots, \lambda_s \in P$ such that
			$\text{wt}(V) \subset \bigcup_{j=1}^s (\lambda_j - R_{+})$,
			
			\vskip 2mm
			
			\item[(iii)]
			${\mathtt e}_{il} (V_{\mu}) \subset V_{\mu + l \alpha_{i}} \ \ \text{for all} \
			\mu \in P, \ (i,l) \in I^{\infty}$.
			
		\end{enumerate}
		
	}
\end{definition}

\vskip 2mm

For $v \in V \setminus \{0\}$, define
\begin{equation*}
	d_{il}^{\vee} = \max\{n \in \Z_{\ge 0} \mid {\mathtt e}_{il}^n v \neq 0 \}
	= \min \{n \in \Z_{\ge 0} \mid {\mathtt e}_{il}^{n+1} v = 0\}.
\end{equation*}

\vskip 2mm

By the Definition of $d_{il}^\vee$, it is easy to check the following lemma holds.

\begin{lemma}
{\rm	For a given $k\in\mathbf Z_{\geq 0}$, the space 
	\begin{equation}\label{eq:Vik}
		V_{il}^{<k}:=\{v\in V\mid d_{il}^{\vee}(v)<k\}=\{v\in V\mid \mathtt e_{il}^k(v)=0\}
	\end{equation}
	is a $\mathbf k$-subspace of $V$.}
\end{lemma}

\vskip 1mm

\begin{definition} \label{def:upper perfect basis} 
	{\rm
		Let $(V, \{{\mathtt e}_{il}\}_{(i,l) \in I^{\infty}})$ be a weak upper perfect space.
		A basis $\mathbb{B}$ of $V$ is called an {\it upper perfect basis} if

\vskip 2mm		
		\begin{enumerate}
			
			\item[(a)] $\mathbb{B} = \sqcup_{\mu \in P} \mathbb{B}_{\mu}$, where
			$\mathbb{B}_{\mu} = \mathbb{B} \cap V_{\mu}$,
			
			\vskip 2mm
			
			\item[(b)] for any $(i,l) \in I^{\infty}$, there exists a map
			$\mathbf{E}_{il}: \mathbb{B} \rightarrow \mathbb{B} \cup \{0\}$ such that
			for all $b \in \mathbb{B}$ there exists $c \in \mathbf{k}^{\times}$ satisfying
		\vskip 2mm	
			\begin{enumerate}
				
				\item[(i)] $\mathtt{e}_{il}(b) - c \, \mathbf{E}_{il}(b) \in V_{il}^{^{<d_{il}^\vee(b)-1}} $,
		
		\vskip 2mm		
				\item[(ii)] if $d_{il}^{\vee}(b) = 0$, then $\mathbf{E}_{il}(b) = 0$,
		
		\vskip 2mm		
				\item[(iii)] if $d_{il}^{\vee}(b) >0$, then $\mathbf{E}_{il}(b) \in \mathbb{B}$.
				
			\end{enumerate}
			
			\vskip 2mm
			
			\item[(c)] $\mathbf{E}_{il}(b) = \mathbf{E}_{il}(b')$ for all $(i,l) \in I^{\infty}$ implies $b = b'$.
			
		\end{enumerate}

	}
\end{definition}

\vskip 3mm

A $P$-graded vector space $V = \bigoplus_{\mu \in P} V_{\mu}$ is called an {\it upper perfect
	space} if it has an upper perfect basis.
Let $\mathbb{B}$ be an upper perfect basis of $V$.
We define the maps $\mathbf{F}_{il}: \mathbb{B} \rightarrow \mathbb{B}\cup \{0\}$,
$\text{wt}:\mathbb{B} \rightarrow P$, $\varepsilon_{i}, \varphi_{i}:\mathbb{B} \rightarrow \Z\cup \{-\infty\}$
by

\begin{equation*}
	\begin{aligned}
		& \mathbf{F}_{il}(b) = \begin{cases} b' \ \ & \text{if} \ b = \mathbf{E}_{il}(b'), \\
			0 \ \ & \text{if} \ b \notin \mathbf{E}_{il}(\mathbb{B}),
		\end{cases} \\
		& \text{wt}(b) = \begin{cases} \mu \ \ & \text{if} \ b \in \mathbb B_{\mu}, \\
			0 \ \ & \text{if} \ b \notin \mathbb B_\mu,
		\end{cases} \\
		& \varepsilon_{i}(b) = \begin{cases} d_{i}^{\vee}(b) \ \ & \text{if} \ i \in I^{\text{re}}, \\
			0 \ \ & if \ i \in I^{\text{im}},
		\end{cases} \\
		& \varphi_{i}(b) = \varepsilon_{i}(b) + \langle h_{i}, \text{wt}(b) \rangle.
	\end{aligned}
\end{equation*}

\vskip 2mm

Then $(\mathbb{B}, \mathbf{E}_{il}, \mathbf{F}_{il}, \text{wt}, \varepsilon_{i}, \varphi_{i})$
is an abstract crystal, called the {\it upper perfect graph} of $\mathbb{B}$.

\vskip 4mm

\begin{lemma} \label{lem:lwcore} 
	{\rm
		
		Let $(V, \{\mathtt{e}_{il}\}_{(i,l)\in I^{\infty}})$ be an upper perfect space with an
		upper perfect basis $\mathbb{B}$.
		Then the following statements hold.
		
		\vskip 2mm
		
		\begin{enumerate}
			\item[(a)]	For $b\in\mathbb B$, $k\geq 1$, we have
			\begin{equation*}
				\mathtt e_{il}^k(b)-c\mathbf E_{il}^k(b)\in V_{il}^{<d_{il}^\vee(b)-k}.
			\end{equation*}
			
			\item[(b)] For any $k \in \Z_{\ge 0}$, we have
			\begin{equation*}
				\ker{\mathtt e}_{il}^k = \bigoplus_{b \in \mathbb{B} \atop d_{il}^{\vee}(b)<k} \mathbf{k}b.
			\end{equation*}
			
			\item[(c)] For any $k \in \Z_{\ge 0}$ and $(i, l) \in I^{\infty}$, the image of
			$\{b \in \mathbb{B} \mid d_{il}^{\vee}(b) = k\}$ is a basis of
			$\ker \mathtt{e}_{il}^{k+1} \big/ \ker \mathtt{e}_{il}^{k}$.
			
		\end{enumerate}
		
	}
\end{lemma}

\begin{proof}
	For (a), we use induction on $k$. 
	If $k=1$, then our assertion follows from Definition \ref{def:upper perfect basis} (b).
	We assume 
	\begin{equation}\label{eq:induction}
		\mathtt e_{il}^{k-1}(b)-c\mathbf E_{il}^{k-1}(b)\in V_{il}^{<d_{il}^\vee(b)-k+1}.
	\end{equation}
	
	By the definition of $d^\vee_{il}$, we have $d_{il}^\vee(\mathtt e_{il}v)=d_{il}^\vee(v)-1$. Therefore, we have 
	
	\begin{equation}\label{eq:eilV}
		\mathtt e_{il}V_{il}^{d_{il}^\vee(b)-k+1}\subset V_{il}^{d_{il}^\vee(b)-k}.
	\end{equation}
	
	By \eqref{eq:induction} and \eqref{eq:eilV}, we have
	\begin{equation}\label{eq:eilkb}
		\mathtt e_{il}^{k}(b)-c\mathbf e_{il}\mathbf E_{il}^{k-1}(b)\in V_{il}^{\ell_{il}(b)-k}.
	\end{equation}
	
	By Definition \ref{def:upper perfect basis}, we have
	\begin{equation}\label{eq:eilEil}
		\mathtt e_{il}(\mathbf E_{il}^{k-1}(b))-c'\mathbf E_{il}^k(b)\in V_{il}^{d_{il}^\vee(\mathbf E_{il}^{k-1}(b))-1}.
	\end{equation}

\vskip 2mm
	
	Again by the definition of $d^\vee_{il}$, we have $d_{il}^\vee(\mathbf E_{il}(v))\leq d_{il}^\vee(v)-1$.

\vskip 2mm
	
	Then we obtain
	\begin{equation*}
		d_{il}^\vee(\mathbf E_{il}^k(b))
		\leq d_{il}^{\vee}(\mathbf E_{il}^{k-1}(b))-1\leq \cdots\leq d_{il}^\vee(b)-k.
	\end{equation*}

\vskip 2mm
	
	By the definition of $V_i^{<k}$, we have
	\begin{equation*}
		V_{il}^{d_{il}^\vee(\mathbf E_{il}^{k-1}(b))-1}\subset V_{il}^{d_{il}^\vee(\mathbf E_{il}^{k-2}(b))-2}\subset \cdots \subset V_{il}^{d_{il}^\vee(b)-k}.
	\end{equation*}
	
	Then from \eqref{eq:eilEil}, we have
	\begin{equation}\label{eq:eilEilk-1}
		\mathtt e_{il}(\mathbf E_{il}^{k-1}(b))-c'\mathbf E_{il}^k(b)\in V_{il}^{d_{il}^\vee(b)-k}.
	\end{equation}
	
	Combining \eqref{eq:eilkb} and \eqref{eq:eilEilk-1}, we get the desired result.

\vskip 2mm
	
	The assertion (b) follows from \eqref{eq:Vik} and (c) follows from (b).
\end{proof}

\vskip 3mm

Now we will discuss the connection between lower perfect bases and upper perfect bases.
Let $V = \bigoplus_{\mu \in P} V_{\mu}$ be a weak lower perfect space together with a
family of endomorphisms
$\mathtt{f}_{il}: V \rightarrow V$ $((i,l) \in I^{\infty})$.
Set
\begin{equation}\label{eq:Vvee}
	V^{\vee}:= \bigoplus_{\mu \in P} V_{\mu}^{\vee},
	\ \ \text{where} \ \ V_{\mu}^{\vee} = \Hom_{\mathbf{k}}(V_{\mu}, \mathbf{k}).
\end{equation}

\vskip 1mm

By abuse of
notation, let $\langle \ , \ \rangle: V \times V^{\vee} \rightarrow \mathbf{k}$ denote the
canonical pairing and define the endomorphisms $\mathtt{e}_{il}:V^\vee \rightarrow V^\vee$
$((i,l) \in I^{\infty})$ by
\begin{equation*}
	\langle u, \mathtt{e}_{il}v \rangle = \langle \mathtt{f}_{il} u, v \rangle
	\ \ \text{for} \ u \in V, \ v \in V^{\vee}.
\end{equation*}

\vskip 2mm

As usual, for a basis $B$ of $V$, we denote by $B^{\vee} \subset V^{\vee}$ the dual basis of $B$.

\vskip 2mm

\begin{proposition} \label{prop:dual}
	
	\vskip 2mm
	
	{\rm
		
		Let $(V, \{\mathtt{f}_{il}\}_{(i,l)\in I^{\infty}})$ be a weak lower perfect space.
		Let $\mathbf{B}$ be a basis of $V$ and $\mathbf{B}^{\vee}$ be the dual basis of $\mathbf{B}$.
		
		\vskip 2mm
		
		\begin{enumerate}
			
			\item[(a)] $\B$ is a lower perfect basis of $V$ if and only if $\mathbf{B}^{\vee}$ is an upper perfect
			basis of $V^{\vee}$.
			
			\vskip 2mm
			
			\item[(b)] Let $\mathbf{B}$ be a lower perfect basis of $V$ and denote the canonical bijection
			$\mathbf{B} \longrightarrow \mathbf{B}^{\vee}$ by $b \longmapsto b^{\vee}$.
			Then we have
			\begin{equation*}
				d_{il}(b) = d_{il}^{\vee}(b^{\vee}), \ \ (\mathbf{e}_{il}(b))^{\vee} = \mathbf{E}_{il}(b^{\vee})
				\ \ \text{for all} \ b \in \mathbf{B}, \ (i,l) \in I^{\infty}.
			\end{equation*}
			
		\end{enumerate}
		
	}
	
\end{proposition}

\vskip 2mm

\begin{proof} \ Since the main argument is similar to that of  \cite[Proposition 7.3]{KKKS},
	we will only give a sketch of the proof.
	
	\vskip 3mm
	
	We will prove {\it if} part.
	Assume that $\mathbf{B}^{\vee}$ is an upper perfect basis of $V^{\vee}$.
	Then
	\begin{equation*}
		\ker \mathtt{e}_{il}^k= \bigoplus_{b^{\vee} \in \mathbf{B}^{\vee}
			\atop d_{il}^{\vee}(b^{\vee}) < k} \mathbf{k}\, b.
	\end{equation*}

	Since
	\begin{equation*}
		\mathtt{f}_{il}^k V = (\ker \mathtt{e}_{il}^k)^{\perp}
		= \{u \in V \mid \langle u, \ker \mathtt{e}_{il}^k \rangle = 0 \},
	\end{equation*}
	we obtain
	$$\mathtt{f}_{il}^k V = \bigoplus_{b \in \B \atop d_{il}^{\vee}(b^{\vee}) \ge k} \mathbf{k} \,b.$$
	It follows that $b \in \mathtt{f}_{il}^k V$ if and only if $d_{il}^{\vee}(b^{\vee}) \ge k$.
	Hence we obtain
	
	\begin{equation*}
		d_{il}(b) = d_{il}^{\vee}(b^{\vee}) \ \ \text{and} \ \
		\mathtt{f}_{il}^k V = \bigoplus_{b \in \B \atop d_{il}(b) \ge k} \mathbf{k}\, b.
	\end{equation*}
	
	\vskip 2mm
	
	Define the map
	
	\begin{equation*}
		\mathbf{f}_{il}: \B \longrightarrow \B \cup \{0\} \ \ \text{by} \ \
		b \longmapsto (\mathbf{F}_{il}(b^{\vee}))^{\vee} \ \ \text{for} \
		b \in \B, \ (i,l) \in I^{\infty}.
	\end{equation*}
	
	\vskip 3mm
	
	Let $b \in \B$ and suppose $d_{il}(b) = k$.
	We will show that
	\begin{equation*}
		\mathtt{f}_{il}(b) - c \, \mathbf{f}_{il}(b) \in \mathtt{f}_{il}^{k+2} V
		\ \ \text{for some} \ c \in \mathbf{k}^{\times}.
	\end{equation*}
	
	\vskip 2mm
	
By Lemma \ref{lem:lwcore}, the image of $\mathbf{B}_{k}^{\vee}
	= \{b^{\vee} \in \B^{\vee} \mid d_{il}^{\vee}(b) = k\}$ is a basis of
	$\ker \mathtt{e}_{il}^{k+1} \big/ \ker \mathtt{e}_{il}^{k}$.
	Since $(\mathtt{f}_{il}^k V \big/ \mathtt{f}_{il}^{k+1}V)^{\vee}
	\overset{\sim} \rightarrow \ker \mathtt{e}_{il}^{k+1} \big/ \ker \mathtt{e}_{il}^k$,
	the image of the dual $\B_{k}$ of $\B_{k}^{\vee}$ is a basis of
	$\mathtt{f}_{il}^k V \big/ \mathtt{f}_{il}^{k+1} V$.
	
	\vskip 3mm
	
	By our assumption, the map
	$$\mathtt{e}_{il}: \ker \mathtt{e}_{il}^{k+2} \big/ \ker \mathtt{e}_{il}^{k+1}
	\longrightarrow \ker \mathtt{e}_{il}^{k+1} \big/ \ker \mathtt{e}_{il}^{k}$$
	induces an injective map
	$$\mathbf{E}_{il}: \{b^{\vee} \in \mathbf{B}_{k+1}^{\vee} \mid
	\mathbf{E}_{il}(b^{\vee}) \neq 0 \} \hookrightarrow \mathbf{B}_{k}^{\vee}.$$
	
	\vskip 2mm
	
	By duality, one can see that the map
	$$\mathtt{f}_{il}: \mathtt{f}_{il}^{k} V \big/ \mathtt{f}_{il}^{k+1} V
	\longrightarrow \mathtt{f}_{il}^{k+1} V \big/ \mathtt{f}_{il}^{k+2} V$$
	induces a map $\B_{k} \rightarrow \B_{k+1}$ up to a constant multiple.
	Hence for any $b \in \B_{k}$, we have
	$$\mathtt{f}_{il}(b) - c \, \mathbf{f}_{il}(b) \in
	\mathtt{f}_{il}^{k+2}V \ \ \text{for some} \  c \in \mathbf{k}^{\times}.$$
	It follows that $\B$ is a lower perfect basis.
	
	\vskip 3mm
	
	The converse can be proved by a {\it dual} argument.
\end{proof}

\vskip 2mm

Recall the definition of $e'_{il}$ in \eqref{eq:eP=R}, one can show that there exists a unique non-degenerate symmetric bilinear form $(\ ,\ )_K$ on $U^-_q(\g)$ which is given by
\begin{equation} \label{eq:bilinearU}
	(\mathbf{1}, \mathbf{1})_{K} = 1, \ \
	(\mathtt{b}_{il}S, T)_{K} = (S, e_{il}'T)_{K} \ \ \text{for} \ S, T \in U_{q}^{-}(\g).
\end{equation}

\vskip 2mm

We set
\begin{equation}\label{eq:Uvee}
	U_{q}^{-}(\g)^\vee:= \bigoplus_{\mu \in P} U_{q}^{-}(\g)_{-\mu}^{\vee},
	\quad U_{q}^{-}(\g)_{-\mu}^{\vee} = \Hom_{\mathbf{k}}(U_{q}^{-}(\g)_{-\mu}, \mathbf{k}),
\end{equation}
where $U_{q}^{-}(\g)_{-\mu}$ is defined in \eqref{eq:weight space U}.

\vskip 2mm

By the non-degenerate symmetric bilinear form $(\ ,\ )_K$ \eqref{eq:bilinearV} on $V(\lambda)$  and the restricted dual space $V(\lambda)^\vee$ defined in \eqref{eq:Vvee}, we can identify $V(\lambda)$ with the dual space $V(\lambda)^\vee$.

\vskip 2mm

Similarly, by the non-degenerate symmetric bilinear form $(\ ,\ )_K$ \eqref{eq:bilinearU} on $U_{q}^{-}(\g)$  and the restricted dual space $U_{q}^{-}(\g)^\vee$ defined in \eqref{eq:Uvee}. We can identify $U_{q}^{-}(\g)$ with the dual space $U_{q}^{-}(\g)^\vee$.

\vskip 2mm

To summarize the above discussion, we obtain the following theorem.

\vskip 2mm

\begin{theorem} \label{thm:final} \hfill
	
	\vskip 2mm
	
	{\rm
		
		\begin{enumerate}
			
			\item[(a)] The upper global basis $\B(\infty)^{\vee}$ is an upper perfect basis of $U_{q}^{-}(\g)$.
			
			\vskip 2mm
			
			\item[(b)] For $\lambda \in P^{+}$, the upper global basis $\B(\lambda)^{\vee}$ is an upper perfect
			basis of $V(\lambda)$.
			
			\vskip 2mm
			
			\item[(c)] Every upper perfect graph of $U_{q}^{-}(\g)$ is isomorphic to $B(\infty)$.
			
			\vskip 2mm
			
			\item[(d)] For $\lambda \in P^{+}$, every upper perfect graph of $V(\lambda)$ is
			isomorphic to $B(\lambda)$.
			
		\end{enumerate}
	}

\end{theorem}

\vskip 2mm

\begin{proof} \ The statements (a), (b) are immediate consequences of Proposition \ref{prop:exist lps} and Proposition \ref{prop:dual}.
	The statements (c), (d) follow from Proposition \ref{prop:lower crystal},
	Corollary \ref{cor:lower crystal} and Proposition \ref{prop:dual}.
\end{proof}


\vskip 10mm


\begin{thebibliography}{7}
\addtolength{\itemsep}{0.8mm}

\bibitem{BeK}
A. Berenstein, D. Kazhdan,
{\em Geometric and unipotent crystals II: from unipotent bicrystals to
	crystal bases}, Contemp. Math. {\bf 433} (2007), 13--88.

\bibitem{Bozec2014b} T. Bozec, {\em Quivers with loops and perverse sheaves}, Math.
Ann. {\bf 362} (2015), 773--797.

\bibitem{Bozec2014c} T. Bozec, {\em Quivers with loops and generalized
	crystals}, Compositio Math. {\bf 152} (2016), 1999--2040.

\bibitem{Dr85}
V. G. Drinfeld,
{\em Hopf algebras and the quantum Yang-Baxter equation},
Soviet Math. Dokl. {\bf 32} (1985), 254--258.

\bibitem{FHKK} Z. Fan, S. Han, S.-J. Kang, Y. R. Kim, {\em Crystal bases and canonical bases for quantum Borcherds-Bozec algebras}, preprint (2024), arXiv:2211.02859.

\bibitem{GZ85}
I.~Gelfand, A.~Zelevinsky, {\em Polytopes in the pattern space and canonical bases for the irreducible representations of $gl_3$}, Funct. Anal. Appl. {\bf 19} (1985), 72--75.

\bibitem{GZ86}
I.~Gelfand, A.~Zelevinsky,  {\em Canonical basis in irreducible representations of gl3 and its applications}, Group theoretical methods in physics {\bf 2} (1986), 127--146.

\bibitem{GZ96}
I.~Gelfand, A.~Zelevinsky,
{\em Multiplicities and regular bases for $gl_n$}, in "Group theoretical methods in Physics," Proc. of the third seminar {\bf 2} (1996), 22--31.

\bibitem{HK02}
J.~Hong, S.-J. Kang, \emph{Introduction to Quantum Groups and Crystal Bases}, Grad. Stud. Math. {\bf 42}, Amer. Math. Soc.,
Provodence, RI, 2002.

\bibitem{Jimbo85} M. Jimbo.
\emph{A $q$-difference equation of $U(\g)$ and the Yang-Baxter equation},
Lett. Math. Phys. {\bf 10} (1985), 63--69.

\bibitem{KKKS}
B. Kahng, S.-J. Kang, M. Kashiwara, U. Suh,
{\em Dual perfect bases and dual perfect graphs},
Mosc. Math. J. {\bf 15} (2015), 319--335.

\bibitem{Kang95} S.-J. Kang,
\emph{Quantum deformations of generalized Kac-Moody algebras and their Modules},
J. Algebra {\bf 175} (1995), 1041--1066.

\bibitem{Kang18} S.-J. Kang, {\em Ringel-Hall algebra construction of quantum Borcherds-Bozec algebras}, J. Algebra {\bf 498} (2018),
344--361.

\bibitem{KKO13}
S.-J. Kang, M. Kashiwara, S-j. Oh, {\em Supercategorification of quantum Kac-Moody algebras},  Adv. Math. {\bf 242} (2013), 116--162.

\bibitem{KKO14}
S.-J. Kang, M. Kashiwara, S-j. Oh, {\em Supercategorification of quantum Kac-Moody algebras II},  Adv. Math. {\bf 265} (2014), 169--240.

\bibitem{KK20} S.-J. Kang, Y. R. Kim, {\em Quantum Borcherds-Bozec algebras and their integrable representations}, J. Pure Appl. Algebra {\bf 224} (2020), 106388.

\bibitem{KOP11} S.-J. Kang, S.-j. Oh, E. Park,
{\em  Perfect bases for integrable modules over generalized Kac-Moody algebras},
Algebr. Represent. Theory {\bf 14} (2011), 571--587.

\bibitem{KOP12} S.-J. Kang, S.-j. Oh, E. Park,
{\em 
Categorification of quantum generalized Kac–Moody algebras and crystal bases},
Internat. J. Math. {\bf 23} (2012), 1250116.


\bibitem{Kashi90} M.~Kashiwara,
\emph{Crystallizing the q-analogue of universal enveloping algebra}, Commun. Math. Phys.
{\bf 133} (1990), 249--260.

\bibitem{Kashi91}
M.~Kashiwara.
\emph{On crystal bases of the {$Q$}-analogue of universal enveloping algebras}, Duke Math. J. {\bf 63} (1991), 465--516.


\bibitem{Lu21} M. Lu, {\em Quantum Borcherds-Bozec algebras via semi-derived Ringel-Hall algebras},
Proc. Amer. Math. Soc. {\bf 151} (2023), 2759--2771.

\bibitem{Lus90} G. Lusztig,
\emph{Canonical bases arising from quantized enveloping algebras}, J. Amer. Math. Soc. {\bf 3}
(1990), 447--498.

\bibitem{Lus91} G. Lusztig,
\emph{Quivers, perverse sheaves and enveloping algebras}, J. Amer. Math. Soc. {\bf 4} (1991), 365--421.
\end{thebibliography}
\end{document}